\documentclass{article}
\usepackage[utf8]{inputenc}
\usepackage{amsmath,amsthm, amssymb, amsfonts}
\usepackage{fullpage}
\usepackage[pdfencoding=auto]{hyperref}
\usepackage{enumerate}
\usepackage{dsfont} 
\usepackage{wasysym}
\usepackage{comment}
\usepackage{authblk}
\usepackage{mathtools}
\usepackage[all]{xy}
\usepackage{microtype}
\usepackage{bm}

\newtheorem{theorem}{Theorem}[section]
\newtheorem*{theorem*}{Theorem}
\newtheorem{remark}[theorem]{Remark}
\newtheorem{definition}[theorem]{Definition}
\newtheorem{proposition}[theorem]{Proposition}
\newtheorem{lemma}[theorem]{Lemma}
\newtheorem{corollary}[theorem]{Corollary}

\newtheorem*{remark*}{Remark}
\newtheorem*{proposition*}{Proposition}

\numberwithin{equation}{section}
\allowdisplaybreaks[1]

\newcommand{\sconc}{\odot}
\newcommand{\cc}{{\rm cc}}
\newcommand{\hccp}[2]{\widehat\Pi^{\cc,#1} \if\relax\detokenize{#2}\relax\else [#2] \fi} 
\newcommand{\ccp}[3]{\Pi^{\cc,#1}_{#2}  \if\relax\detokenize{#3}\relax\else [#3] \fi} 
\newcommand{\nnhhpcc}[4]{\widehat\Phi^{\cc,#1}_{#2,#3} \if\relax\detokenize{#4}\relax\else [#4] \fi} 
\newcommand{\hhpcc}[3]{\widehat\Pi^{\cc,#1}_{#2} \if\relax\detokenize{#3}\relax\else [#3] \fi} 
\newcommand{\CC}{Clenshaw--Curtis }
\newcommand{\Cheb}{{Cheby\v{s}ev} } 
\newcommand{\Chebdash}{{Cheby\v{s}ev}-} 
\DeclareMathOperator{\ifft}{IFFT}
\newcommand{\dist}[1]{\operatorname{dist}_{#1}}
\DeclareMathOperator{\spann}{span \,}
\DeclareMathOperator*{\esssup}{ess\!\sup}

\newcommand{\norm}[2][]{\| #2 \|_{#1}} 
\newcommand{\snorm}[2][]{| #2 |_{#1}} 
\newcommand{\normc}[2][]{\left\| #2 \right\|_{#1}} 
\newcommand{\snormc}[2][]{\left| #2 \right|_{#1}}

\newcommand{\hI}{\hat{I}}
\DeclareMathOperator{\Id}{Id}

\newcommand{\depth}{L}
\newcommand{\size}{M}
\newcommand{\first}{\operatorname{fi}}
\newcommand{\last}{\operatorname{la}}
\newcommand{\sizefirst}{M_{\operatorname{fi}}} 
\newcommand{\sizelast}{M_{\operatorname{la}}} 
\newcommand{\Parallel}[1]{\mathrm{P}\left(#1\right)} 
\newcommand{\dParallel}[1]{\mathrm{FP}\left(#1\right)}

\newcommand{\Ecal}{\mathcal{E}}
\newcommand{\Gcal}{\mathcal{G}}

\newcommand{\Tcal}{\mathcal{T}}

\newcommand{\Z}{\mathbb{Z}}
\newcommand{\N}{\mathbb{N}}

\newcommand{\R}{\mathbb{R}}
\newcommand{\IC}{\mathbb{C}}
\newcommand{\bbP}{\mathbb{P}}

\newcommand{\bsbeta}{{\boldsymbol \beta}}

\newcommand{\bsp}{{\boldsymbol p}}

\newcommand{\bsv}{{\boldsymbol v}}

\newcommand{\bsone}{{\boldsymbol{1}}} 

\newcommand{\eps}{\varepsilon}

\newcommand{\der}[1]{{D^{#1}}}
\newcommand{\realiz}{\mathrm{R}}
\newcommand{\dd}{\mathrm{d}}
\newcommand{\du}[1][]{d_{#1}}
\newcommand{\Cu}[1][]{C_{#1}}
\newcommand{\Nsing}{N_{\rm sing}}

\newcommand{\Wwe}[3]{W^{#1}_{#2,#3}}

\newcommand{\Hwe}[2]{H^{#1}_{#2}}

\begin{document}

\title{
Deep ReLU networks and high-order finite element methods
\\
II: \Cheb emulation
}

\author[$\dagger$]{Joost A. A. Opschoor} 
\author[$\dagger$]{Christoph Schwab}

\affil[$\dagger$]{\footnotesize Seminar for Applied Mathematics, ETH
  Z\"{u}rich, HG G57.1, R\"{a}mistrasse 101, CH--8092 Z\"urich, Switzerland
  \newline \texttt{joost.opschoor@sam.math.ethz.ch,\;
    christoph.schwab@sam.math.ethz.ch}}

\date{\today}

\maketitle

\abstract{
We show expression rates and stability in Sobolev norms 
of deep feedforward ReLU neural networks (NNs) 
in terms of the number of parameters defining the NN
for continuous, piecewise polynomial functions, 
on arbitrary, finite partitions $\Tcal$ of a bounded interval $(a,b)$.
Novel constructions of ReLU NN surrogates encoding function approximations
in terms of \Cheb polynomial expansion coefficients are developed which
require fewer neurons than previous constructions.
\Cheb coefficients can be computed easily
from the values of the function in the \CC points
using the inverse fast Fourier transform.
Bounds on expression rates and stability are obtained that are superior to those of
constructions based on ReLU NN emulations of monomials 
as considered in \cite{OPS2020,MontYangDu21}.
All emulation bounds are explicit in terms of the (arbitrary)
partition of the interval, the target emulation accuracy and
the polynomial degree in each element of the partition.
ReLU NN emulation error 
estimates are provided for various classes of functions and norms, 
commonly encountered in numerical analysis. 
In particular,
we show
exponential ReLU emulation rate bounds for analytic functions with point singularities
and develop an interface between
Chebfun approximations and constructive ReLU NN emulations.



\section{Introduction}
\label{sec:intro}

The use of Deep Neural Networks (DNNs for short) 
as approximation methods for the 
numerical solution of Partial Differential Equations (PDEs for short)
has received considerable attention in recent years.
We mention only the ``PiNNs'' and the ``Deep Ritz'' methodologies,
and their variants.

One question which arises in this context is the DNN
``expressivity'', which could be labeled as approximation power
of the DNN with specific architecture (comprising e.g. depth,
width and activation function),
for particular sets of functions which arise as solution
components (``solution features'' in the parlance of deep learning)
of PDEs in science and in engineering.

In recent years, \emph{numerous results} 
on analytic and approximation properties of deep ReLU NNs 
have appeared.
We only provide an incomplete review, 
with particular attention to results of direct relevance to the present paper.
Previous works focused on
particular classes of univariate functions, e.g \cite{PGEB2018}.
It has also been observed that DNNs, in particular the presently 
considered ReLU activated DNNs, can emulate large classes of well-established
approximation architectures, 
such as refinable functions (wavelets) 
(see, e.g., \cite{BGKP2017,DHP21} and the references there),
splines and high-order polynomial functions \cite{OPS2020}. 
Recently, also rather wide classes of Finite-Element approximation spaces on
regular, simplicial partitions of polytopal domains 
in $\R^d$ have been emulated by ReLU DNNs \cite{LODSZ2022}.
Further results, not of direct relevance here, on approximation rate bounds
for deep ReLU NNs include complexity bounds in terms of the number of affine
pieces \cite{HinzGeer19}, 
expression rate bounds for parametric transport \cite{LaakPhPReLUTrnsp},
and high-dimensional classification 
(see \cite{petersen2021optimal} and the references there).

This paper analyzes expression rates for deep NN emulation 
of splines and high-order polynomials, as in \cite{OPS2020}.
Distinct from \cite{OPS2020},
\Cheb polynomials are used as a basis for spaces of univariate polynomials.
This implies better stability properties of the resulting NN
surrogates, and allows, as we show, 
to obtain more efficient neural network approximations. 
In Section \ref{sec:Contrib},
we detail the benefits of the present, \Cheb polynomial-based approach 
and compare it with previous work, including \cite{OPS2020}.

\subsection{Layout}
\label{sec:Layout}
The structure of this text is as follows:
In Section \ref{sec:Contrib}, 
we summarize the contributions of this paper.
We first recall in Section \ref{sec:conclinterp}
the benefits of using \Cheb polynomials for polynomial interpolation.
ReLU NN approximations based on the conversion of \Cheb interpolants 
are introduced in Section \ref{sec:conclemul}, 
and bounds on the error incurred by conversion are given.
Quantitative improvements of our method with respect to previous work
are stated in Section \ref{sec:conclquant}, in particular
how the \Chebdash based approach delivers the same relative pointwise 
emulation accuracy with a smaller network size.
The relation between our approach and that used in the 
Chebfun software package \cite{chebguideone}
is discussed in Section \ref{sec:conclchebfun}.
The main body of this work starts in Section~\ref{sec:prelim}.
In that section, we recapitulate classical facts from
approximation theory, and from orthogonal polynomials.
Section~\ref{sec:CPwLoned} introduces the spline spaces for which we prove
emulation bounds. 
Section~\ref{sec:functionsponed} formalizes certain scales of weighted function
spaces on intervals $I$, whose members are piecewise smooth, up to 
possibly a finite number of point singularities at points in $I$ 
or at its boundary.
Section~\ref{sec:cconed} recalls polynomial interpolation in the \CC points,
its stability,
and the expression of such interpolants in the basis of \Cheb polynomials
using the inverse fast Fourier transform.
All these concepts introduced in Section \ref{sec:prelim}
are used to construct and analyze,
in Section~\ref{sec:basicnn},
ReLU DNN approximations of univariate polynomials 
and continuous, piecewise polynomial functions.
Using ReLU DNN approximations of univariate, piecewise polynomial functions
to emulate 
$h$-, $p$- and $hp$-finite element methods
provides us 
in Section \ref{sec:FESpaces}
with DNN approximation rates
for various classes of univariate functions.
In particular, we obtain in Section \ref{sec:hpFEM}
improved exponential convergence rates for the approximation of 
piecewise smooth, weighted Gevrey regular functions with point singularities.
Section \ref{sec:Concl} concludes the paper
and discusses extensions of the current results to other activation functions
and to spiking neural networks.
\subsection{Notation}
\label{sec:Notat}
We denote by $C>0$ a generic, positive constant
whose numerical value may be different at each appearance, 
even within an equation. 
The dependence of $C$ on parameters 
is made explicit when necessary, e.g. $C(\eta,\theta)$.

For a finite set $S$ we denote its cardinality by $\snorm{S}$.
For a subset $S\subset\N_0$ and $k\in\N_0$, 
we define $\mathds{1}_{S}(k) = 1$ if $k\in S$, and $\mathds{1}_{S}(k) = 0$ otherwise.
We write $\dist{}(x,y) := \snorm{x-y}$ for all $x,y\in\R$.

For $p\in\N_0$, 
we denote the space of polynomials
of degree at most $p$ by 
$\bbP_p = \spann\{x^j: j\in\N_0, j\leq p\}$, 
with the convention $\bbP_{-1} := \{0\}$.
For an interval $I\subset\R$, 
we will sometimes write $\bbP_p(I)$
to indicate that we consider polynomials as functions on $I$.
For $k\in\N_0$, 
the univariate \Cheb polynomial (of the first kind) 
normalized such that $T_k(1)=1$ is denoted by $T_k$.
With slight abuse of notation we write expressions of the form 
$\left( \sum_i (a_i)^r \right)^{1/r}$ for $r\in[1,\infty]$,
where $a_i\in\R$ for all $i$,
by which we mean $\sup_i a_i$ in case $r=\infty$. 
Similarly, when writing $\left( \int_D \snorm{f(x)}^r \dd x \right)^{1/r}$ 
for a domain $D$ with $f\in L^r(D)$ and $r\in[1,\infty]$,
we mean $\esssup_{x\in D} \snorm{f(x)}$ in case $r=\infty$.

For vector spaces $U,W$ and a bounded linear operator $P: U \to W$,
the operator norm of $P$ is $\norm[U,W]{ P }$.

In connection with neural networks, we shall also invoke the 
realization $\mathrm R$, parallelization $\mathrm P$ and 
the full parallelization $\mathrm{FP}$ in Section~\ref{sec:basicnn}.

\section{Contributions}
\label{sec:Contrib}
%
The main technical contributions of the present paper are 
\emph{improved bounds on the emulation error and neural network complexity}
(as compared to the first part of this work \cite{OPS2020}) 
of deep ReLU NN emulations
of polynomials and continuous, piecewise polynomial functions.
Emulations by deep ReLU NNs which are developed in the proofs 
of the main results are constructive, 
based on point evaluations and on (piecewise) \Cheb expansions.
Improvements with respect to previous work \cite{OPS2020} 
are realized in terms of stability, DNN size vs. accuracy, and efficient construction.

In this section we give an overview of these benefits.
To keep the technical details to a minimum,
we consider in this introduction only ReLU NN emulations of polynomial interpolants on $[-1,1]$.
The analysis given here directly generalizes to ReLU NNs 
based on piecewise polynomial interpolants 
as shown in Section \ref{sec:relupwpolynom}.

Our \emph{constructive ReLU NN emulation} of a given function $g \in C^0([-1,1])$ 
uses nodal interpolation in the \CC nodes in $[-1,1]$.
Specifically, 
denote the Lagrangian polynomial interpolant of $g$ of degree $p\in\N$ by $g_p$
(this corresponds to $\hccp{p}{g}$ in the notation of Section \ref{sec:cconed}).
The main difference of the presently proposed 
construction with respect to \cite{OPS2020} is the use of 
FFT-based methods for \Cheb interpolation.
This allows for fast, numerical construction of the ReLU emulation of a given
continuous function, and for favorable stability properties of the construction:
the sum of the absolute values of the \Cheb coefficients of a polynomial
is much smaller than that of its Taylor coefficients.
Enhanced numerical stability of \Cheb polynomial based DNNs has recently 
been reported in \cite{TLYpre2019}.

Upon a recap on \Cheb expansions 
and polynomial interpolation in the \CC points
in Section \ref{sec:conclinterp},
Section \ref{sec:conclemul} introduces
ReLU NN emulations based on these interpolants
together with bounds on their error and their network depth and size.
The improvements with respect to previous work 
afforded by a \Chebdash based approach
are highlighted in Section \ref{sec:conclquant}.
In Section \ref{sec:conclchebfun},
we relate our approach to that of the Chebfun software package.

\subsection{\Cheb interpolation}
\label{sec:conclinterp}
%
We elaborate on the error and the complexity of interpolation
and assume for simplicity that 
$g \in C^0([-1,1])$
admits a \Cheb expansion $g = \sum_{j=0}^\infty a_j T_j$
with absolutely summable \Cheb coefficients $(a_j)_{j=0}^\infty$.
Uniform pointwise convergence of a \Cheb series is known to be a weaker condition
(\cite[Theorem 3.1]{Rivlin1974}).
It holds true under a logarithmic condition on the modulus of continuity of $g$, 
i.e. $\lim_{n\to\infty}\omega(g;1/n)\log(n) = 0$, cf. \cite[Theorem 3.4]{Rivlin1974}.
A sufficient condition
for absolute summability of $(a_j)_{j=0}^\infty$
is that the first derivative of $g$ is of bounded variation.
This is in particular the case for functions in the Sobolev space
$W^{2,1}((-1,1))$ and also for all functions realized by ReLU NNs.
For such functions, 
it is shown in \cite[Theorem 7.2]{Trefethen2020}
that the \Cheb coefficients $(a_j)_{j\in\N}$ 
satisfy $\snorm{a_j} \leq C V j^{-2}$, 
with $V$ the variation of $g'$.

Focusing in this discussion on error bounds in $L^\infty((-1,1))$
(see Section \ref{sec:conclquant} for the full range of treated Sobolev norms),
the error $\norm[L^\infty((-1,1))]{ g - g_p }$
is determined by the smoothness of $g$, 
which can be expressed in terms of the decay of its \Cheb coefficients.
We denote $g_p = \sum_{j=0}^p c_j T_j$.
Here, the coefficients $(c_j)_{j=0}^p$ 
may differ from $(a_j)_{j=0}^p$ due to aliasing
and can be computed from function values in the \CC nodes
using the inverse fast Fourier transform,
as will be discussed in Proposition \ref{prop:cconed} Item \eqref{item:ccifftoned}.
They depend linearly on the function $g$ which is to be approximated 
and are given explicitly in terms of $(a_j)_{j=0}^\infty$ 
by \cite[Theorem 4.2]{Trefethen2020}.
The interpolation error can be estimated by
\begin{align*}
\norm[L^\infty((-1,1))]{ g - g_p }
	\leq &\, \sum_{j=0}^p \snorm{ a_j - c_j } \norm[L^\infty((-1,1))]{ T_j }
		+ \sum_{j=p+1}^\infty \snorm{ a_j } \norm[L^\infty((-1,1))]{ T_j }
	\\
	\leq &\, \sum_{j=0}^p \snorm{ a_j - c_j }
		+ \sum_{j=p+1}^\infty \snorm{ a_j }
	\leq   2 \sum_{j=p+1}^\infty \snorm{ a_j }
	,
\end{align*}
where the last step used \cite[Theorem 4.2]{Trefethen2020}.
Thus, if for some $\alpha>0$ holds that
$\snorm{ a_j } \leq C j^{-1-\alpha}$ for all $j\in\N$, 
then $\norm[L^\infty((-1,1))]{ g - g_p } \leq C p^{-\alpha}$.
As it is well-known, cf. e.g. \cite{Rivlin1974} and \cite[Sections 7 and 8]{Trefethen2020},
Sobolev smoothness of the represented function
is characterized by the decay of the \Cheb coefficients.
For example, as mentioned before,
$\alpha=1$ for functions 
whose first derivative is of bounded variation
due to \cite[Theorem 7.2]{Trefethen2020},
which includes functions in
$W^{2,1}((-1,1))$ 
and realizations of ReLU NNs.

The moderate growth of the Lebesgue constant of the \CC nodes 
in terms of the polynomial degree $p$
(with respect to the $L^\infty((-1,1))$-norm it is at most 
 $\tfrac{2}{\pi} \log(p+1) +1$,
 see Proposition \ref{prop:cconed} Item \eqref{item:cclebconstoned})
implies favorable numerical stability of Lagrange interpolation
in the \CC nodes.

This allows us to build a polynomial $u_p$ 
that is approximately equal to $g$
by interpolating an approximation $u$ of $g$ whose 
\Cheb coefficients are also absolutely summable,
e.g. a ReLU NN approximation of $g$.
It follows that 
$$
\begin{array}{rcl}
\norm[L^\infty((-1,1))]{ g - u_p } 
	&\leq& \norm[L^\infty((-1,1))]{ g - g_p } 
		+ \norm[L^\infty((-1,1))]{ g_p - u_p }
\\
	&\leq& \norm[L^\infty((-1,1))]{ g - g_p } 
		+ ( \tfrac{2}{\pi} \log(p+1) +1 ) \norm[L^\infty((-1,1))]{ g - u },
\end{array}
$$
i.e. in addition to the error $g - g_p$
there is a second term in which 
the approximation error $g-u$ is magnified by at most a factor $( \tfrac{2}{\pi} \log(p+1) +1 )$.

The computational complexity of interpolation includes 
at most $O(p)$ function evaluations and an additional $O(p (1+\log(p) ) )$ operations
for the execution of the fast Fourier transform.

If $u$ is the realization of a ReLU NN from Section \ref{sec:Polynomials},
then the \Cheb coefficients $(c_j)_{j=0}^p$ of $g_p$
can simply be read off from the output layer network weights,
exactly, and with complexity $p+1$.
In this case, there is no need to interpolate $u$ by $u_p$
in order to use its \Cheb coefficients.

\subsection{\Chebdash based ReLU emulation}
\label{sec:conclemul}
Assuming at hand the interpolant $g_p$,
for example through its chebfun object representation
(see Section~\ref{sec:conclchebfun} below),
we denote its relative error by 
$\eps_p := \norm[L^\infty((-1,1))]{ g - g_p }$ $/ \norm[L^\infty((-1,1))]{ g }$
and assume for simplicity that $\eps_p\leq1$.

The ReLU NN emulation $\tilde{g}_p$ of the chebfun object $g_p$ 
is constructed in two steps.
First, a ReLU DNN emulating the \Cheb polynomials $(T_j)_{j=1}^p$
is constructed;
this is done in Lemma~\ref{lem:relupolynominduction} below.
We denote its outputs by $(\widetilde{T}_j)_{j=1}^p$.
These ReLU emulations of the $T_j$ incur an error.
For all $0 < \delta < 1$, 
the constructed ReLU NN emulating 
$(\widetilde{T}_j)_{j=1}^p$
with 
$\norm[L^\infty((-1,1))]{ T_j - \widetilde{T}_j } \leq \delta < 1$ for all $j=1,\ldots,p$
has a network size which we show to be bounded from above by 
$C (p (1+\log(p)) + p \log(1/\delta) )$,
with $C>0$ independent of $p\in \N$ and of $\delta$.
The second step is to compute in the output layer the linear combination
$\tilde{g}_p := \sum_{j=0}^p c_j \widetilde{T}_j$.
We have $T_0 \equiv 1$,
so we can define $\widetilde{T}_0 := 1$ 
and add the constant term $c_0 \widetilde{T}_0$ 
as a bias in the output layer, without error.
Similarly, the linear polynomial $T_1$ can be emulated exactly,
i.e. $\widetilde{T}_1 := T_1$.
The emulation error of $\tilde{g}_p$ can therefore be estimated by
\begin{align*}
\norm[L^\infty((-1,1))]{ g_p - \tilde{g}_p }
	\leq \sum_{j=0}^p \snorm{c_j} \norm[L^\infty((-1,1))]{ T_j - \widetilde{T}_j }
	\leq \delta \sum_{j=2}^p \snorm{c_j}
.
\end{align*}
The favorable conditioning of the \Cheb polynomials 
will allow us to show in Lemma \ref{lem:chebclinftyoned} 
that
$\sum_{j=2}^p \snorm{c_j} \leq p^4 \norm[L^\infty((-1,1))]{ g_p }$
(we actually will prove this for all polynomials of degree at most $p$),
which implies that 
for all $0<\tilde{\eps}_p<1$
a relative emulation error tolerance
$\tilde{\eps}_p \geq \norm[L^\infty((-1,1))]{ g_p - \tilde{g}_p } / \norm[L^\infty((-1,1))]{ g_p }$
can be achieved by setting
$\delta = \tilde{\eps}_p / p^4$.
Substituting into the bound on the network size gives
$C (p (1+\log(p)) + p \log(1/\tilde{\eps}_p ) )$.}

Choosing $\tilde{\eps}_p\leq\eps_p$ provides
a ReLU NN emulation $\tilde{g}_p$ of the \Cheb interpolant $g_p$ of $g$
of degree $p$.
Its pointwise accuracy is proportional to that furnished by $g_p$, 
where $g_p$ can be constructed efficiently, e.g. by Chebfun algorithms:
\begin{align*}
\norm[L^\infty((-1,1))]{ g - \tilde{g}_p } 
	\leq &\, \norm[L^\infty((-1,1))]{ g - g_p } + \tilde{\eps}_p \norm[L^\infty((-1,1))]{ g_p }
	\\
	\leq &\, \norm[L^\infty((-1,1))]{ g - g_p } + \tilde{\eps}_p ( \norm[L^\infty((-1,1))]{ g } + \norm[L^\infty((-1,1))]{ g - g_p } )
	\\
	\leq &\, ( \eps_p + \tilde{\eps}_p (1+\eps_p) ) \norm[L^\infty((-1,1))]{ g }
	\leq 3 \eps_p \norm[L^\infty((-1,1))]{ g }
.
\end{align*}

Furthermore, when the \Cheb coefficients are known, 
the computational complexity for computing $\tilde{g}_p$
equals the size of the resulting NN,
which is at most $C( p (1+\log(p) ) + p \log(1/\tilde{\eps}_p) )$.
\subsection{Quantitative improvements over previous results}
\label{sec:conclquant}
We compare our \Chebdash based construction 
with the approach in \cite{OPS2020},
which was based on a monomial expansion 
$g_p = \sum_{j=0}^p t_j x^j$,
and a ReLU DNN emulation of the monomials $(x^j)_{j=1}^p$.
We denote its outputs by $(\widetilde{X}_j)_{j=1}^p$.
For all $0 < \delta < 1$, 
the ReLU NN emulating $(\widetilde{X}_j)_{j=1}^p$ constructed in 
\cite[Proofs of Lemma 4.5 and Proposition 4.2]{OPS2020}
with $\norm[L^\infty((-1,1))]{ x^j - \widetilde{X}_j } \leq \delta$ for all $j=1,\ldots,p$
has a network size which is bounded by 
$C (p (1+\log(p)) + p \log(1/\delta) )$.
Again, the second step is to compute in the output layer a linear combination
$\tilde{g}^{\mathrm{M}}_p := \sum_{j=0}^p t_j \widetilde{X}_j$.
The emulation error of $\tilde{g}^{\mathrm{M}}_p$ can be estimated by
\begin{align*}
\norm[L^\infty((-1,1))]{ g_p - \tilde{g}^{\mathrm{M}}_p }
	\leq \sum_{j=0}^p \snorm{t_j} \norm[L^\infty((-1,1))]{ x^j - \widetilde{X}_j }
	\leq \delta \sum_{j=2}^p \snorm{t_j}
.
\end{align*}
Bad conditioning of the monomials implies that
there does not exist $c>0$ such that for 
any integer $p\geq2$ and all polynomials of degree $p$, 
their Taylor coefficients can be estimated by 
$\sum_{j=2}^p \snorm{t_j} \leq \norm[L^\infty((-1,1))]{ g_p } p^c$.
On the contrary, 
upper bounds on 
$\sum_{j=2}^p \snorm{t_j} / \norm[L^\infty((-1,1))]{ g_p }$ 
which hold uniformly for all polynomials of degree $p$
generally grow exponentially with $p$.
For example, 
for $p\geq 2$ and $g = T_p$ 
we have $g_p = T_p$,
$\norm[L^\infty((-1,1))]{ g_p } = 1$ and $t_p = 2^{p-1}$,
which means that 
$\sum_{j=2}^p \snorm{t_j} \geq \snorm{t_p} = 2^{p-1}$
and thus 
$\sum_{j=2}^p \snorm{t_j} \geq 2^{p-1} \norm[L^\infty((-1,1))]{ g_p }$,
which implies that 
a relative emulation error 
$\tilde{\eps}_p := \norm[L^\infty((-1,1))]{ g_p - \tilde{g}^{\mathrm{M}}_p } / \norm[L^\infty((-1,1))]{ g_p }$
requires
$\delta \leq \tilde{\eps}_p / 2^{p-1}$.
Substituting into the bound on the network size gives
a term of the order $C p^2$.
In \cite[Proposition 4.6]{OPS2020},
it was shown that there exist constants $c,C>0$
such that it is sufficient to choose
$\delta = \tilde{\eps}_p C c^p$,
resulting in the upper bound
$C (p^2 + p \log(1/\tilde{\eps}_p ) )$
on the network size.\footnote{
The exponential growth of Taylor coefficients
holds more generally than for the example of $T_p = g = g_p$.
For the polynomial interpolant $g_p = \sum_{j=0}^p c_j T_j$ 
of any continuous function $g$,
we can expand $(T_j)_{j=0}^p$ in their monomial expansions
to obtain the monomial expansion $g_p = \sum_{j=0}^p t_j x^j$.
Using again that the $p$'th coefficient of $T_p$ equals $2^{p-1}$
and observing that the only term with the power $x^p$ comes from $c_j T_j$,
we see that $t_p = 2^{p-1} c_p$.
}

Compared to \cite{OPS2020}, 
we provide error bounds in terms of a wider range of Sobolev spaces
$W^{s,r}(I)$ for $0\leq s \leq 1$ and $1\leq r \leq \infty$
on arbitrary bounded intervals $I = (a,b)$,
keeping track of the natural scaling of Sobolev norms 
as a function of the interval length $b-a$.

\subsection{Chebfun}
\label{sec:conclchebfun}

Advantages of representing a function through its \Cheb coefficients
have also been leveraged in the so-called ``Chebfun''
software package, cf. {\tt http://www.chebfun.org}.
There too, 
functions are represented via \Cheb coefficients of the interpolant
in the \CC nodes, in a so-called ``chebfun'' object
(cf. \cite[Section 1.1]{chebguideone}).\footnote{
The label ``Chebfun'' (with capital C) refers to the software package \cite{chebguideone},
and chebfun, with lowercase c, to an object within this software 
(representing a piecewise polynomial function through its \Cheb coefficients).}
Chebfun has a broad range of functionalities and acts 
directly on the \Cheb coefficients.

Our DNN emulation result in Proposition \ref{prop:relupwpolynom} 
allows in particular what could be called 
``Chebfun - Neural Network'' interoperability via
transfer of all functionalities of Chebfun
to the presently considered ReLU NN context. 
Specifically:
 
(i) Given a ReLU NN, a chebfun object can be created through interpolation 
(with the Chebfun function {\tt chebfun}).
In case the NN was constructed as in Proposition \ref{prop:relupwpolynom},
then no interpolation is necessary.
The \Cheb coefficients can be read off directly from the output layer weights.
 
(ii) Any Chebfun functionality can be applied to the obtained chebfun object,
including those functions whose output is again a chebfun object.

(iii) Our NN emulation results
provide ReLU NN emulations of the output chebfun object.
$$
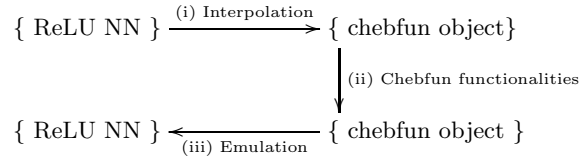
\xymatrix{
\{ \text{ ReLU NN } \} 
\ar^-{\text{(i) Interpolation}}[r] 
&
*+[r]{
\{ \text{ chebfun object} \} }
\ar^{\text{(ii) Chebfun functionalities}}[d]
\\
\{ \text{ ReLU NN } \} 
&
*+[r]{
\{ \text{ chebfun object } \} }
\ar^-{\text{(iii) Emulation}}[l]
}
$$
Similarly, ReLU NN functionalities may enhance Chebfun via
$$
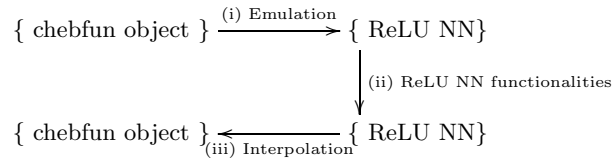
\xymatrix{
\{ \text{ chebfun object } \} 
\ar^-{\text{(i) Emulation}}[r] 
&
*+[r]{
\{ \text{ ReLU NN} \} }
\ar^{\text{(ii) ReLU NN functionalities}}[d]
\\
\{ \text{ chebfun object } \} 
&
*+[r]{
\{ \text{ ReLU NN} \} }
\ar^-{\text{(iii) Interpolation}}[l]
}
$$
e.g., when
the functions afford only low regularity in classical
H\"older- or Besov spaces, but exhibit self-similar structure,
as is the case e.g. with fractal function systems.
We refer to the discussion in \cite[Section~7.3.1]{DauDeVetal22}: 
for certain fractal function classes with low Sobolev regularity and
\emph{dense singular support}, 
deep ReLU NNs can still afford exponential approximation rates in terms of the NN size, 
whereas polynomial approximations furnished by Chebfun
will converge only at low, algebraic rates in terms of the polynomial degree.

\section{Preliminaries}
\label{sec:prelim}

We prepare the presentation of our main results and their proofs
by recapitulating notation and classical results on direct and inverse
approximation rates of univariate polynomials.

\subsection{Polynomial inverse inequalities}
\label{sec:inverseineq}

The classical polynomial inverse inequality due to the Markovs, reads as follows.
\begin{lemma}[Markov inequality]
\label{lem:markovoned}
For all $k\in\N$ and $r\in[1,\infty]$ there exists a constant $C(k,r) > 0$ 
such that for all $p\in\N$ and $\widehat{v} \in \bbP_p([-1,1])$
\begin{align}
\label{eq:markovoned}
\normc[L^r((-1,1))]{ \tfrac{\dd^k}{\dd x^k} \widehat{v} } 
	\leq &\, C(k,r) \norm[L^r((-1,1))]{ \widehat{v} } \prod_{i=0}^{k-1} (p-i)^2
	.
\end{align}
For $k=1$, the constant satisfies the uniform bound $C(1,r) \leq 6e^{1+1/e}$ for all $r\in[1,\infty]$.
For all $k\in\N$ it holds that $C(k,r)\leq C(1,r)^k$.
In addition, 
with respect to the $L^\infty$-norm we have the sharper constants
$C(1,\infty) = 1$ and $C(2,\infty) = 1/3$.
\end{lemma}

\begin{proof} 
A combination of \cite[Section 3]{HST1937}, \cite{Goetgheluck1990},
\cite[Chapter 4, Theorem 1.4]{DeVoreLorentz1993}
and
\cite[Chapter 4, Equation (12.2)]{DeVoreLorentz1993},
see \cite[Lemma 2.1.1]{JOdiss} and its proof.
\end{proof}

We shall also require 
the following inverse inequality from, e.g., 
\cite[Chapter 4, Theorem 2.6]{DeVoreLorentz1993}.
It holds for all 
$0 < q \leq r \leq \infty$, $p\in\N$ 
and 
$\widehat{v} \in\bbP_p([-1,1])$:
\begin{align}
\label{eq:invineqoned}
\normc[L^r((-1,1))]{ \widehat{v} } 
	\leq \normc[L^q((-1,1))]{ \widehat{v} } \left((1+q) p^2\right)^{1/q - 1/r}
.
\end{align}
\subsection{Piecewise polynomials}
\label{sec:CPwLoned}
For $-\infty<a<b<\infty$,
consider a partition $\Tcal$ of the interval $I = (a,b)$ into $N\in\N$ elements, with nodes 
$a = x_{0} < x_{1} < \ldots 
< x_{N-1} < x_{N} = b$, 
elements $I_i = (x_{i-1},x_{i})$ and element sizes
$h_i =  x_{i}-x_{i-1}$ for $i\in\{1,\ldots,N\}$. 
Let $h  =  \max_{i\in\{1,\ldots,N\}} h_i$.
For a polynomial degree distribution 
$\bsp = (p_i)_{i\in\{1,\ldots,N\}} \in \N^N$ 
on $\Tcal$, we define 
$p_{\max} = \max_{i\in\{1,\ldots,N\}} p_i$ and the corresponding approximation space
\begin{equation}\label{eq:hpFESpc}
S_{\bsp}(I,\Tcal)  =  \{ v \in C^0(I): 
                v|_{I_i} \in \bbP_{p_i}(I_i) \text{ for all } i \in \{1,\ldots,N\} \} .
\end{equation}

Our $hp$-approximations in Section
\ref{sec:hpFEM}
will be based on geometrically graded partitions
of $I = (0,1)$ which are refined towards $x=0$.
In the notation from \cite{OPS2020}, these are defined as follows:
for $N\in\N$ and $\sigma\in(0,1)$, 
let
$x_{0} \coloneqq 0$ and 
$x_{i}\coloneqq \sigma^{N-i}$ for $i\in\{1,\ldots,N\}$.
Then $\Tcal_{\sigma,N}$ is the partition of $I$ 
into $N$ intervals $\{I_{\sigma,i}\}_{i=1}^{N}$,
where $I_{\sigma,i}\coloneqq(x_{i-1}, x_{i})$.
%
\subsection{Weighted function spaces on $I$}
\label{sec:functionsponed}
We introduce finite order Sobolev spaces and analytic classes of
Sobolev spaces in the open, bounded interval $I$,
with and without weight.
These spaces will be employed to describe 
the regularity of functions 
in ReLU DNN emulations.

\subsubsection{Finite order Sobolev spaces}
\label{sec:fctspcfin}
For a bounded open interval $I\subset\R$ 
and $u:I\to\R$, 
for all $k\in\N$
we write
\begin{gather*}
\der{k}u 
	= \frac{\dd^{k}u}{\dd x^{k}}
	,
\end{gather*}
and also, for all $r\in[1,\infty)$ and all $k\in\N_0$,
\begin{gather*}
\norm[W^{k,r}(I)]{u} = \left( \sum_{k'=0}^k \normc[L^r(I)]{ \der{k'}u }^r \right)^{1/r}
,
\qquad
\snorm[W^{k,r}(I)]{u} = \normc[L^r(I)]{ \der{k}u }
,
\\
\norm[W^{k,\infty}(I)]{u} = \max_{k'=0}^k \normc[L^\infty(I)]{ \der{k'}u }
,
\qquad
\snorm[W^{k,\infty}(I)]{u} = \normc[L^\infty(I)]{ \der{k}u }
,
\end{gather*}
and for $r=2$ we write
$H^k(I) := W^{k,2}(I)$.

Although we are mainly interested in real-valued functions, 
some of the error estimates are given in 
Lebesgue spaces of complex-valued functions
defined on some open set $\Ecal \subset \IC$ containing $I$.
For such spaces, the codomain will be indicated explicitly, 
e.g. $L^\infty(\Ecal;\IC)$.
\subsubsection{Finite order weighted Sobolev spaces}
\label{sec:FinOrdWgtSobSpc}
To allow for functions with point singularities, but otherwise smooth,
we next introduce a continuous weight function
defined in terms of a finite number of singular points.
We will consider spaces of functions which may be singular in given points 
$A_1,\ldots,A_{\Nsing} \in \R$ for $\Nsing\in\N$.
The $A_i$ will usually (but not always) be assumed to belong to $\overline{I}$.
For a \emph{weight sequence} $\bsbeta = (\beta_1,\ldots,\beta_{\Nsing})\in \R^{\Nsing}$,
we consider the \emph{weight function}
\begin{align}
\label{eq:wedef}
\psi_{\bsbeta}(x) = \prod_{i=1}^{\Nsing} \min\{1,\dist{}(x , A_i) \}^{\beta_i}
,
\qquad
x\in\R
.
\end{align}
If ${\Nsing}=1$, we will write $\beta$ instead of $\bsbeta$.

For $r\in[1,\infty)$ and $m,\ell\in\N_0$ with $m\geq \ell$,
the space $\Wwe{m,\ell}{r}{\bsbeta}(I)$ is defined as
the closure of $C^\infty(I)$ with respect to the 
weighted Sobolev norm 
\begin{align*}
\norm[\Wwe{m,\ell}{r}{\bsbeta}(I)]{u}
	&= \left( \norm[W^{\ell-1,r}(I)]{u}^r
		+ \sum_{k=\ell}^m \normc[L^r(I)]{\psi_{\bsbeta+k-\ell} \der{k}u }^r \right)^{1/r}
	,
	&&
	\text{ if } \ell>0,
	\\
\norm[\Wwe{m,0}{r}{\bsbeta}(I)]{u}
	&= \left( \sum_{k=0}^m \normc[L^r(I)]{\psi_{\bsbeta+k} \der{k}u }^r \right)^{1/r}
	,
	&&
	\text{ if } \ell=0.
\end{align*}
The corresponding seminorm equals
\begin{align*}
\snorm[\Wwe{m,\ell}{r}{\bsbeta}(I)]{u}
	&= \normc[L^r(I)]{\psi_{\bsbeta+m-\ell} \der{m}u }
.
\end{align*}
For $r=\infty$, these norms are understood w.r. to the $L^\infty(I)$- and $\| \circ \|_{\infty}$-norms.
In the Hilbertian case $r=2$, we write 
$\Hwe{m,\ell}{\bsbeta}(I) := \Wwe{m,\ell}{2}{\bsbeta}(I)$.
For large weight exponents $\bsbeta>0$ in \eqref{eq:wedef},
the weight function $\psi_{\bsbeta}$ is small close to the singular points $A_i$,
i.e. the function $u$ and its derivatives 
are allowed to grow strongly close to the singularity.

\subsubsection{Finite order fractional Sobolev spaces}
\label{sec:FinOrdFrc}
In Section \ref{sec:relupwpolynom} we will briefly comment on 
error bounds with respect to norms of fractional Sobolev spaces 
$W^{s,r}(I)$ for $s\in(0,1)$ and $r\in[1,\infty]$ on a bounded interval $I = (a,b)$,
for $-\infty<a<b<\infty$.
We recall that
\begin{align*}
\begin{cases}
\norm[W^{s,r}(I)]{ u } 
	:= \left( \norm[L^r(I)]{ u }^r + \int_{I\times I} \tfrac{ \snorm{ f(x) - f(y) }^r }{ \snorm{ x - y }^{sr+1} } \dd x \dd y 
		\right)^{1/r}
	& \text{ if } r < \infty,
	\\
\norm[W^{s,\infty}(I)]{ u } 
	:= \max\left\{ \norm[L^\infty(I)]{ u } , \esssup_{x,y\in I, x\neq y} \tfrac{ \snorm{ f(x) - f(y) } }{ \snorm{ x - y }^{s} } \right\}
	& \text{ if } r = \infty.
\end{cases}
\end{align*}

\subsubsection{Weighted Gevrey classes}
\label{sec:fncspcAn}
\emph{Weighted Gevrey classes}
are sets of functions on $I=(0,1)$ 
which may be singular in $x=0$,
but are smooth (not necessarily analytic) away from $0$.
We refer to \cite{FS2020,CPS2011} 
and the references there for such spaces.
They contain as particular cases also
weighted analytic functions (which satisfy 
Equation \eqref{eq:defGevrey} below for $\delta=1$).
We define for $\delta\geq1$, $\ell\in\N_0$ and $\beta\in(0,1)$
the Gevrey class
$\Gcal^{\ell,\delta}_\beta(I)$ to be the set of functions $u$ in 
$\bigcap_{k\geq\ell} \Hwe{k,\ell}{\beta}(I)$
for which there exist constants $\Cu[u],\du[u]>0$ such that
\begin{align}
\label{eq:defGevrey}
\forall k\geq\ell:\quad
	\snormc[\Hwe{k,\ell}{\beta}(I)]{u}
	\leq \Cu[u] \du[u]^{k-\ell} ((k-\ell)!)^\delta.
\end{align}
Rates of their ReLU NN approximations will be studied in Section \ref{sec:hpFEM}.
\subsection{Polynomial interpolation in \CC points}
\label{sec:cconed}
Due to its central role in our NN constructions,
we recall basic properties of interpolation in the \CC points,
including stability bounds and how  
\Cheb coefficients of the interpolant
can be computed from function values in the \CC points
using the inverse fast Fourier transform (IFFT).
This is done in Proposition \ref{prop:cconed}.
In the proof of Proposition \ref{cor:relupolynomH10}, 
we need estimates on the \Cheb coefficients of polynomials in $\bbP_p((-1,1))$
in terms of the $L^\infty((-1,1))$-norm of the polynomial.
These are provided in Lemma \ref{lem:chebclinftyoned}.

For $p\in\N$ we consider the $p+1$ \CC points\footnote{
In the literature, the nomenclature for these points is not consistent.
Besides the term ``\CC points'', these points 
are sometimes referred to as ``\Cheb points of the second kind'',
or even simply as ``\Cheb points''.
They are not to be confused with the $p$ points that are the roots of the \Cheb polynomial $T_p$,
which lie in $(-1,1)$. These points
are sometimes also referred to as ``\Cheb points'', or as ``\Cheb points of the first kind''.}
in $[-1,1]$,
which are the extrema of the \Cheb polynomial $T_p$.
We denote them by
\begin{align}
\label{eq:ccpointsoned}
\hat{x}^{\cc,p}_{j} := \cos(j\pi/p),
\qquad
j\in\{0,\ldots,p\}
.
\end{align}
It will be convenient to extend this definition to all $j\in\Z$,
even though this does not introduce new points, as
$\{\hat{x}^{\cc,p}_{j} : j\in\{0,\ldots,p\} \} = \{\hat{x}^{\cc,p}_{j} : j\in\Z \}$.
In particular, for $j=p+1,\ldots,2p$ it holds that $\hat{x}^{\cc,p}_{j} = \hat{x}^{\cc,p}_{2p-j}$.

The nodal interpolant 
$\hccp{p}{\widehat{v}} \in \bbP_p([-1,1])$ 
in the \CC grid 
of a continuous function $\hat{v} \in C^0([-1,1])$
is defined by
\begin{align}
\label{eq:ccinterpoloned}
\widehat{v} ( \hat{x}^{\cc,p}_{j} ) = \hccp{p}{\widehat{v}} ( \hat{x}^{\cc,p}_{j} ),
\qquad
\text{ for all } j\in\{0,\ldots,p\}
.
\end{align}

For $-\infty<a<b<\infty$ and $I=[a,b]$, 
we will consider the images of $(\hat{x}^{\cc,p}_{j})_{j\in\{0,\ldots,p\}}$
under the affine map $F:x\mapsto \tfrac{a+b}{2} + \tfrac{b-a}{2} x$,
and denote them by
\begin{align}
\label{eq:ccinterpolonedaffine}
x^{\cc,p}_{I,j} = \tfrac{a+b}{2} + \tfrac{b-a}{2} \cos(j\pi/p),
\qquad
\text{ for all } j\in\{0,\ldots,p\}
.
\end{align}
For every $v\in C^0(\overline{I})$ we denote by 
$\ccp{p}{I}{v} := ( \hccp{p}{v \circ F} ) \circ F^{-1} \in \bbP_p(I)$
the interpolant in these points.
	
The following proposition recalls properties of $\hccp{p}{}$.

\begin{proposition}
\label{prop:cconed}
The Lagrange interpolation operator $\hccp{p}{}: C^0([-1,1]) \to \bbP_p([-1,1])$ defined in \eqref{eq:ccinterpoloned}
 has the following properties:
\begin{enumerate}[(i)]
\item \label{item:ccprojoned}
For all $p\in\N$ and $\widehat{v}\in\bbP_p([-1,1])$, it holds that $\hccp{p}{\widehat{v}} = \widehat{v}$.
\item \label{item:cclebconstoned}
For all $p,s\in\N$, $p\geq s$,
the \emph{Lebesgue constants} of $\hccp{p}{}$ satisfy the following bounds:
\begin{align*}
\norm[L^\infty,L^\infty]{\hccp{p}{}} 
	:= &\, \norm[L^\infty((-1,1)),L^\infty((-1,1))]{\hccp{p}{}} 
	\leq (\tfrac{2}{\pi}\log(p+1) +1),
	\\
\norm[W^{s,\infty},W^{s,\infty}]{\hccp{p}{}}	
	:= &\, \norm[W^{s,\infty}((-1,1)),W^{s,\infty}((-1,1))]{\hccp{p}{}} 
	\leq p^{2s} \norm[L^\infty,L^\infty]{\hccp{p}{}}
.
\end{align*}
\item \label{item:ccifftoned}
For all $p\in\N$ and $\widehat{v}\in C^0([-1,1])$, 
the \Cheb coefficients of the interpolant $\hccp{p}{\widehat{v}}$
can be computed with the inverse discrete Fourier transform:
\begin{align*}
\hccp{p}{\widehat{v}} = \sum_{k\in\{0,\ldots,p\}} \widehat{v}_{k,p} T_k,
\end{align*}
where
\begin{align*}
\widehat{v}_{k,p} := &\, 2^{\mathds{1}_{\{1,\ldots,p-1\}}(k)} 
		\left( \ifft \left( \left( \widehat{v} \left( \hat{x}^{\cc,p}_{j} \right) \right)_{j\in\{0,\ldots,2p-1\}} 
		\right) \right)_{k}
	\\
	= &\, \frac{2^{\mathds{1}_{\{1,\ldots,p-1\}}(k)}}{2p} \sum_{j\in\{0,\ldots,2p-1\}} 
		\widehat{v}( \cos(j\pi/p) ) \cos(kj\pi/p)
.
\end{align*}
\end{enumerate}
\end{proposition}

\begin{proof}
\begin{enumerate}[(i)]
\item 
It follows e.g. from
\cite[Equation (4.9)]{Trefethen2020},
with $0 = a_{n+1} = a_{n+2} = \ldots$ in the notation of \cite{Trefethen2020}.
\item
The first bound is stated e.g. in \cite[Theorem 15.2]{Trefethen2020}.
There also a brief history of this result is provided.
The second estimate 
follows from the first one using Markov's inequality, Lemma \ref{lem:markovoned}.
\item 
This is shown in \cite[Theorem 3.13]{Rivlin1974}.
\end{enumerate}
\end{proof}

To estimate the neural network approximation error in Section \ref{sec:basicnn}, 
it will be important to have a stability bound 
on the sum of the absolute values of the coefficients in Item \eqref{item:ccifftoned}.
In Lemma \ref{lem:chebclinftyoned}, 
we estimate the sum of the absolute values of \Cheb coefficients 
of a polynomial in terms of its $L^\infty((-1,1))$-norm.

\begin{lemma}
\label{lem:chebclinftyoned}
For all $p\in\N$ and $\widehat{v} = \sum_{\ell=0}^p \widehat{v}_\ell T_\ell \in\bbP_p([-1,1])$, 
defining $\widehat{v}_\ell := 0$ for $\ell>p$, 
it holds that
\begin{align}
\label{eq:chebclinftyoned} 
\sum_{\ell\geq 2} |\widehat{v}_\ell| 
	\leq p^4 \min_{\widehat{w}\in\bbP_p([-1,1]): \atop \der{2}\widehat{w} = \der{2}\widehat{v}} 
		\normc[L^\infty((-1,1))]{\widehat{w}}
	.
\end{align}
\end{lemma}
\begin{proof}
We use \cite[Theorem 4.2]{MR2403058} with $k=1$ in the notation of the reference 
(note that slightly sharper estimates are given in \cite[Theorem 2.1]{MR3588586} 
and \cite[Theorem 7.1]{Trefethen2020})
\begin{align*}
|\widehat{v}_\ell| \leq &\, \tfrac{2}{\pi \ell(\ell-1)} 
	\normc[L^1((-1,1))]{\frac{\der{2}\widehat{v}}{\sqrt{1-x^2}}}, 
	\qquad \ell\geq 2,
	\\
\sum_{\ell\geq 2} |\widehat{v}_\ell|
	\leq &\, \tfrac{2}{\pi} \normc[L^1((-1,1))]{\frac{\der{2}\widehat{v}}{\sqrt{1-x^2}}} 
		\sum_{\ell\geq 2} \left( \tfrac{1}{\ell-1} - \tfrac{1}{\ell} \right)
	\leq \tfrac{2}{\pi} \normc[L^1((-1,1))]{\frac{\der{2}\widehat{v}}{\sqrt{1-x^2}}}.
\end{align*}
For $p=1$ both sides of the inequalities vanish.
Using Markov's inequality (Lemma \ref{lem:markovoned}),
we obtain
\begin{align*}
\normc[L^1((-1,1))]{\frac{\der{2}\widehat{v}}{\sqrt{1-x^2}}}
	\leq &\, \normc[L^1((-1,1))]{\frac{1}{\sqrt{1-x^2}}} \normc[L^\infty((-1,1))]{\der{2}\widehat{v}}
	\\
	\leq &\, \pi \tfrac13 p^2 (p-1)^2 \min_{\widehat{w}\in\bbP_p: \atop \der{2}\widehat{w} = \der{2}\widehat{v}} 
		\normc[L^\infty((-1,1))]{\widehat{w}}.
\end{align*}
Together with the previous estimate, 
this shows \eqref{eq:chebclinftyoned}.
\end{proof}

\section{ReLU NN approximation of univariate functions}
\label{sec:basicnn}
In this section we consider NNs with the
ReLU activation function
$
	\varrho: \R \to \R: x \mapsto \max\{0, x\}
$
for the approximation of univariate functions on bounded intervals.
The ReLU NN approximation of continuous, piecewise polynomial functions
allows us to transfer existing results 
on approximation by piecewise polynomial functions
to obtain approximation results for ReLU NNs.
We discuss several such approximation results in Section \ref{sec:FESpaces}.
Combined with \cite[Section 6]{OPS2020},
these also give efficient approximations to high-dimensional, radially symmetric functions,
as developed in \cite[Sections 6.1--6.3]{OPS2020}.
\subsection{ReLU neural network calculus}
\label{sec:NNReLuCalc}
As usual (e.g. \cite{PV2018,OPS2020}), 
we define a NN as an $L$-tuple of weight-bias pairs.
A NN realizes a function, called \emph{realization} of the NN,
as composition of parameter-dependent affine transformations
and a nonlinear activation function $\varrho$.
We briefly recall the NN formalism from \cite[Section 2]{OPS2020}.
\begin{definition}[Neural Network, Realization $\mathrm R$ {\cite[Definition 2.1]{OPS2020}}]
\label{def:NeuralNetworks}
Let $d, L\in \N$. 
A \emph{neural network $\Phi$ with input dimension $d$ and $L$ layers} 
is a sequence of matrix-vector tuples (\emph{weight matrices} and \emph{bias vectors})
\[
    \Phi = \big((A_1,b_1),  (A_2,b_2),  \dots, (A_L, b_L)\big), 
\]
where $N_0 := d$ and $N_1, \dots, N_{L} \in \N$
are called \emph{layer dimensions},
and where $A_\ell \in \R^{N_\ell\times N_{\ell-1}}$ and $b_\ell \in \R^{N_\ell}$
for $\ell =1,...,L$.

For a NN $\Phi$ and an activation function $\varrho: \R \to \R$, 
we define the associated
\emph{realization of the NN $\Phi$} as 
$
 \realiz(\Phi): \R^d \to \R^{N_L} : x\mapsto x_L, 
$
where 
\begin{equation*}
    \begin{split}
        x_0 &:= x, \\
        x_{\ell} &:= \varrho(A_{\ell} \, x_{\ell-1} + b_\ell) \qquad \text{ for } \ell = 1, \dots, L-1,\\
        x_L &:= A_{L} \, x_{L-1} + b_{L}.
    \end{split}
\end{equation*}
Here $\varrho$ acts componentwise on vector-valued inputs, 
$\varrho(y) = (\varrho(y_1), \dots, \varrho(y_m))$ 
for all $y = (y_1, \dots, y_m) \in \R^m$.
We call 
$L(\Phi):= L$ the \emph{number of layers} or \emph{depth} of $\Phi$, 
$M_j(\Phi):= \| A_j\|_{\ell^0} + \| b_j \|_{\ell^0}$ 
the \emph{number of nonzero weights and biases in the $j$-th layer}, 
and
$M(\Phi) := \sum_{j=1}^L M_j(\Phi)$ 
the \emph{number of nonzero weights} or \emph{size} of $\Phi$. 
The number of nonzero weights and biases in the first layer 
is also denoted by $\sizefirst(\Phi)$, 
and the number of those in the last layer also by $\sizelast(\Phi)$.
The first $L-1$ layers,
in which the activation function is applied,
are called \emph{hidden layers}.
\end{definition}
We construct ReLU DNN emulations of 
Finite Element spaces \eqref{eq:hpFESpc} 
from certain fundamental building blocks using a 
\emph{calculus of ReLU NNs} from \cite{PV2018}.
 
The following two propositions introduce parallelizations in which, 
given two networks of equal depth,
one network is constructed which in parallel emulates
the realizations of the two given networks.
In the first proposition, both subnetworks have the same input.
In the second proposition, they have different inputs.

Both propositions contain bounds on the number of nonzero weights in the first and the last layer
that can be derived from the definitions.

\begin{proposition}[Parallelization $\mathrm P$, { \cite[Definition 2.7]{PV2018}, \cite[Proposition 2.3]{OPS2020} }]
\label{prop:parall}
Let $d, L \in \N$ and let $\Phi^1, \Phi^2$ be two NNs 
with $d$-dimensional input and depth $L$.

Then there exists a network 
$\mathrm{P}(\Phi^1, \Phi^2)$ with $d$-dimensional input and $L$ layers, 
called \emph{parallelization of $\Phi^1$ and $\Phi^2$}, 
satisfying
\begin{equation*}
\realiz\left(\mathrm{P}\left(\Phi^1,\Phi^2\right)\right) (x) 
= 
\left(\realiz\left(\Phi^1\right)(x), \realiz\left(\Phi^2\right)(x)\right), 
\qquad\text{ for all } x \in \R^d,
\end{equation*}
$M(\mathrm{P}(\Phi^1, \Phi^2)) = M(\Phi^1) + M(\Phi^2)$, 
$\sizefirst(\mathrm{P}(\Phi^1, \Phi^2)) = \sizefirst(\Phi^1) + \sizefirst(\Phi^2)$ 
and
$\sizelast(\mathrm{P}(\Phi^1, \Phi^2)) $ $= \sizelast(\Phi^1) + \sizelast(\Phi^2)$.
\end{proposition}

We will also use parallelizations of NNs that do not have the same inputs. 
\begin{proposition}[Full parallelization ${\mathrm{FP}}$ {{ \cite[Setting 5.2]{EGJS2021}, \cite[Proposition 2.5]{OPS2020} }}] 
\label{prop:parallSep}
Let $L \in \N$ and let
$\Phi^1,\Phi^2$
be two NNs with the same depth $L$ 
and input dimensions $N^1_0=d_1$ and $N^2_0=d_2$, respectively. 

Then there exists a ReLU NN, denoted by $\mathrm{FP}(\Phi^1, \Phi^2)$, 
with $d = d_1+d_2$-dimensional input and depth $L$, 
called \emph{full parallelization of $\Phi^1$ and $\Phi^2$}, such that 
for all $x = (x_1,x_2) \in \R^d$ with $x_i \in \R^{d_i}, i = 1,2$ 
\begin{align*} 
\realiz\left(\mathrm{FP}\left(\Phi^1,\Phi^2\right)\right) (x_1,x_2) 
	= &\, \left(\realiz\left(\Phi^1\right)(x_1), \realiz\left(\Phi^2\right)(x_2)\right)
	,
\end{align*}
$M(\mathrm{FP}(\Phi^1, \Phi^2)) 
	= M(\Phi^1) + M(\Phi^2)$,
$\sizefirst(\mathrm{FP}(\Phi^1, \Phi^2)) 
	= \sizefirst(\Phi^1) + \sizefirst(\Phi^2)$
and
$\sizelast(\mathrm{FP}(\Phi^1, \Phi^2)) 
	= \sizelast(\Phi^1) + \sizelast(\Phi^2)$.
\end{proposition}

We next recall the concatenation of two NNs.
The matrix multiplications in the definition below
may lead to an undesirably large network size,
if the product of two small or sparse matrices is a large dense matrix
(this applies to $U^1V^{L_2}$ in the definition below).
For this reason, so-called \emph{sparse concatenations} 
will be defined in Proposition \ref{prop:conc}.

\begin{definition}[Concatenation, {{\cite[Definition 2.2]{PV2018}}}]
\label{def:pvconc}
Let $\Phi^1,\Phi^2$ be two NNs 
such that the input dimension of $\Phi^1$ equals the output dimension of $\Phi^2$,
which we denote by $k$.
For $\ell = 1, \ldots, L_1 := \depth(\Phi^1)$ 
we denote the weight matrices and bias vectors of $\Phi^1$ by 
$\{U^{\ell}\}_{\ell}$ and $\{a^{\ell}\}_{\ell}$ 
and for $\ell = 1, \ldots, L_2 := \depth(\Phi^2)$
we denote those of $\Phi^2$ by
$\{V^{\ell}\}_{\ell}$ and $\{b^{\ell}\}_{\ell}$.
Then, we define 
the \emph{concatenation} of $\Phi^1$ and $\Phi^2$ as
\begin{align*}
\Phi^1 \bullet \Phi^2 
	:= &\, \big(
	(V^1,b^1),\ldots,(V^{L_2-1},b^{L_2-1}),
	\\
	&\, (U^1V^{L_2},U^1b^{L_2}+a^1),
	(U^2,a^2),\ldots,(U^{L_1},a^{L_1})
	\big)
. 
\end{align*}
\end{definition}
It follows immediately from this definition that
$\realiz( \Phi^1 \bullet \Phi^2 ) = \realiz( \Phi^1 ) \circ \realiz( \Phi^2 )$ and that
\begin{align}
\label{eq:pvconcdepth}
\depth(\Phi^1 \bullet \Phi^2) = \depth(\Phi^1) + \depth(\Phi^2) - 1
.
\end{align}

\begin{proposition}[{{ Sparse concatenation, \cite[Remark 2.6]{PV2018}, \cite[Proposition 2.2]{OPS2020} }}]
\label{prop:conc}
Let $L_1, L_2 \in \N$, and let 
$\Phi^1$, $\Phi^2$ 
be two ReLU NNs of respective depths $L_1$ and $L_2$ such that $N^1_0 = N^2_{L_2}=: d$, i.e.,
the input layer of $\Phi^1$ has the same dimension as the output layer of $\Phi^2$. 

Then, there exists a ReLU NN $\Phi^1 \sconc \Phi^2$, called 
the \emph{sparse concatenation of $\Phi^1$ and $\Phi^2$}, 
such that $\Phi^1 \sconc \Phi^2$ has depth $L_1+L_2$,
$\realiz(\Phi^1 \sconc \Phi^2) = \realiz(\Phi^1) \circ \realiz(\Phi^2)$, 
and
\begin{gather}
\sizefirst(\Phi^1 \sconc \Phi^2) 
	\leq \begin{cases}
		2 \sizefirst(\Phi^2) & \text{ if }
		L_2 = 1,
		\\
		\sizefirst(\Phi^2) & \text{ else,}
	\end{cases}
	\qquad
\sizelast(\Phi^1 \sconc \Phi^2) 
	\leq \begin{cases}
		2 \sizelast(\Phi^1) & \text{ if }
		L_1 = 1,
		\\
		\sizelast(\Phi^1) & \text{ else,}
	\end{cases}
\nonumber\\
\label{eq:concsize}
M\left(\Phi^1 \sconc \Phi^2\right) \leq M\left(\Phi^1\right) + 
     \sizefirst \left(\Phi^1\right) + \sizelast \left(\Phi^2\right) + M\left(\Phi^2\right) 
  \leq  2M\left(\Phi^1\right)  + 2M\left(\Phi^2\right) .
\end{gather}
\end{proposition}

The following proposition summarizes the properties of
one type of identity NNs based on the ReLU activation function.

\begin{proposition}[{{ \cite[Remark 2.4]{PV2018}, \cite[Proposition 2.4]{OPS2020} }}]
\label{prop:Id}
For every $d,L\in \N$ there exists a ReLU NN 
$\Phi^{\mathrm{Id}}_{d,L}$ with $L(\Phi^{\mathrm{Id}}_{d,L}) = L$, 
$M(\Phi^{\mathrm{Id}}_{d,L}) \leq 2 d L$, 
$\sizefirst(\Phi^{\mathrm{Id}}_{d,L})\leq 2d$ and 
$\sizelast(\Phi^{\mathrm{Id}}_{d,L})\leq 2d$ 
such that 
$\realiz (\Phi^{\mathrm{Id}}_{d,L}) = \mathrm{Id}_{\R^d}$.
\end{proposition}
\subsection{ReLU emulation of polynomials}
\label{sec:Polynomials}
In this section, we construct ReLU NNs 
whose realizations emulate univariate polynomials
of arbitrary degrees on a bounded interval.
We closely follow the approach in \cite[Section 4]{OPS2020},
again exhibiting explicitly the dependence of the DNN depth and size
on the polynomial degree.

Differently from the construction in \cite{OPS2020}, 
for $n\in\N$ we use \Cheb polynomials of the first kind, 
denoted by $\{T_\ell\}_{\ell\leq n}$,
as basis for the space $\bbP_n$ of polynomials of degree $n$.
As was derived for NNs with the RePU activation function 
$\varrho^r$ for $r\in\N$ satisfying $r\geq2$ in \cite{TLYpre2019}, 
and as we will see 
for ReLU NNs in Lemma \ref{lem:relupolynominduction} below, 
they can be approximated efficiently by 
exploiting the \emph{multiplicative three term recursion}
(which is distinct from the additive three-term recurrence relation 
 and not available for other families of orthogonal polynomials on $[-1,1]$)
\begin{align}
\label{eq:chebrecur}
T_{m+n} = 2T_mT_n-T_{|m-n|},
\quad T_0(x) = 1, \quad T_1(x) = x, 
\qquad 
\text{ for all } m,n\in\N_0 \text{ and } x\in\R,
\end{align}
which is related to the addition rule for cosines.\footnote{
For all $m,n\in\Z$ and $\theta\in\R$ it holds that
$\cos((m+n)\theta)
	= \cos(m \theta) \cos (n \theta) - \sin(m \theta) \sin (n \theta)$,
from which we obtain
$\cos((m+n)\theta) + \cos((m-n)\theta)
	= 2 \cos(m \theta) \cos (n \theta)$.
For all $x \in [-1,1]$, 
\eqref{eq:chebrecur} now follows from
$\cos(k \theta) = T_{\snorm{k}}(x)$
for all $k \in \Z$ and $\theta = \arccos x$.}
The advantage of using a \Cheb basis instead of a monomial basis
is that \Cheb approximation is better conditioned.
For example, for functions whose derivative is of bounded variation,
the sum of the absolute values of its \Cheb coefficients 
is bounded in terms of that variation (\cite[Theorem 7.2]{Trefethen2020}).\footnote{\label{fn:chebcoeff}
In \cite[Theorem 7.2]{Trefethen2020}, 
it is shown that the \Cheb coefficients $(a_j)_{j\in\N}$ 
satisfy $\snorm{a_j} \leq C V j^{-2}$, 
where $V$ denotes the variation of the derivative.
A similar estimate is used in the proof of Lemma \ref{lem:chebclinftyoned}.
}
This includes functions in $W^{2,1}((-1,1))$,
for which this variation of its derivative is bounded by the $W^{2,1}((-1,1))$-norm,
and realizations of ReLU NNs.
The sum of the absolute values of the \Cheb coefficients
is often much smaller than the sum of the absolute values of the coefficients
with respect to the monomial basis.
The use of the better conditioned \Cheb basis
allows us to obtain bounds on the network depth and size 
with a better dependence on the polynomial degree
than that in \cite{OPS2020},
because the analysis there used monomial expansions of Legendre polynomials,
the coefficients of which in absolute value grow exponentially with the polynomial degree,
cf. \cite[Equation (4.13)]{OPS2020} and the analysis preceding that equation.

The base result of this section is Proposition \ref{prop:relupolynom}
on the emulation of polynomials on the reference interval $\hI := [-1,1]$,
which is exact in the endpoints of the interval.
For the ReLU DNN approximation of univariate piecewise polynomial functions in Section \ref{sec:relupwpolynom},
we will use NN approximations of polynomials on an arbitrary bounded interval, 
such that the NN realizations are continuous and constant outside the interval of interest.
Such approximations are presented in Corollary \ref{cor:relupolynomH10}.
In both those results, 
the hidden layer weights and biases
are independent of the approximated polynomials.
The output layer weights and biases are the \Cheb coefficients
of the polynomial of interest.
They can be computed easily from the values of the polynomial in the \CC points
using the inverse fast Fourier transform, 
by Proposition \ref{prop:cconed} Item \eqref{item:ccifftoned}.

The main building block for our approximations of piecewise polynomials by ReLU NNs
is the efficient approximation of products introduced in \cite{Yarotsky2017}. 
We here recall a version with $W^{1,\infty}$-error bounds from \cite{SZ2019}
in the notation of \cite{OPS2020}
and in addition recall in Equation \eqref{eq:multone} below from \cite{JOdiss}
that multiplication by $\pm1$ is emulated exactly.

\begin{proposition}[{{ \cite[Propositions 2 and 3]{Yarotsky2017}, \cite[Proposition 3.1]{SZ2019}, \cite[Proposition 6.2.1]{JOdiss} }}]
\label{prop:multiplicationThing}
There exist $C_L, C'_L,$ $C_M, C'_M, C_{\first}, C_{\last} >0$ 
such that
for all $\kappa>0$ and $\delta \in (0,1/2)$, 
there exists a ReLU NN $\widetilde{\times}_{\delta,\kappa}$ 
with input dimension $2$ and output dimension $1$, 
which satisfies
\begin{align*}
& \sup_{|a|, |b| \leq \kappa}\left|ab - \realiz\left(\widetilde{\times}_{\delta, \kappa}\right)(a,b)\right| 
  \leq \delta \text{ and } \nonumber 
\\
&\esssup_{|a|,  |b| \leq \kappa}
\max\left\{
\left|a - {\tfrac{\partial}{\partial b}} \realiz\left(\widetilde{\times}_{\delta, \kappa}\right)(a,b)\right|, 
\left|b - {\tfrac{\partial}{\partial a}} \realiz\left(\widetilde{\times}_{\delta, \kappa}\right)(a,b)\right|
\right\} 
\leq \delta.
\end{align*}
For all $\kappa>0$ and $\delta \in (0, 1/2)$
\begin{align*}
\depth\left(\widetilde{\times}_{\delta, \kappa}\right) 
	\leq &\, C_L \left(\log_2\left( \tfrac{\max\{{\kappa},1\}}{\delta}\right)\right) + C'_L
	,
	\\
\size\left(\widetilde{\times}_{\delta, \kappa}\right) 
	\leq &\, C_M \left(\log_2\left( \tfrac{\max\{{\kappa},1\}}{\delta}\right)\right) + C'_M
	,
	\\
\sizefirst\left(\widetilde{\times}_{\delta, \kappa}\right) 
	\leq &\, C_{\first}
	,
	\qquad
\sizelast\left(\widetilde{\times}_{\delta, \kappa}\right)
	\leq C_{\last}
.
\end{align*}
In addition, for every $a \in \R$
\begin{align}\label{eq:ZeroInZeroOutBasic}
\realiz\left(\widetilde{\times}_{\delta, \kappa}\right)(a,0) 
=
\realiz\left(\widetilde{\times}_{\delta, \kappa}\right)(0,a) = 0,
\end{align}
and for every $a\in\{-1,1\}$ and every $b \in [-\kappa,\kappa]$
\begin{align}\label{eq:multone}
\realiz\left(\widetilde{\times}_{\delta, \kappa}\right)(a,b) 
=
\realiz\left(\widetilde{\times}_{\delta, \kappa}\right)(b,a) = ab.
\end{align}
\end{proposition}
We refer to \cite[Proposition 6.2.1]{JOdiss} for details of the proof,
which is based on results proved in \cite{Yarotsky2017,SZ2019}.
The networks from 
\cite[Propositions 2 and 3]{Yarotsky2017} and \cite[Proposition 3.1]{SZ2019}
have a fixed width and variable depth.
There exist networks with the same realization,
whose width is variable independently of their depth,
which are studied in \cite{PV2018} and a series of works following \cite{SYZ2020}.

We now prove an analogue of \cite[Proposition 4.2]{OPS2020}
on the NN approximation of polynomials on the reference interval $\hI = [-1,1]$,
based on the better conditioned \Cheb expansion of the target polynomial.
It also covers the straightforward extension to the approximation of $N\in\N$ polynomials
by computing them as linear combinations of approximate \Cheb polynomials,
which are emulated only once.
Furthermore, 
the only NN coefficients which depend on the approximated functions
are the \Cheb coefficients, which can be computed efficiently
from point values in the \CC nodes,
as noted in Proposition \ref{prop:cconed} Item \eqref{item:ccifftoned}.

The next result has ReLU DNN expression rate bounds in Lipschitz norm
for polynomials of degree $n$ given in terms of \Cheb expansions.
It has been proved in \cite[Proposition 6.2.2]{JOdiss}, and 
was already used in  \cite[Lemma 3.1]{HOS2022}.
It generalizes \cite[Lemma 3.1]{HOS2022},
which corresponds to the choices of
$N_{\bsv} = n$, ${\tau} = \delta$ and $v_i = T_i$, $i=1,\ldots,n$
in Proposition \ref{prop:relupolynom}.
Those choices imply $C_{\mathrm s} \leq 1$ and $C_{\mathrm n} = n$.

\begin{proposition}[ReLU-Emulation of \Cheb expansions {{\cite[Proposition 6.2.2]{JOdiss}}}]
\label{prop:relupolynom}
There exists a constant $C>0$ such that for all $n,{N_{\bsv}}\in\N$ 
and ${N_{\bsv}}$ polynomials $v_i = \sum_{\ell=0}^n v_{i,\ell} T_{\ell} \in \bbP_n$ for $i=1,\ldots,{N_{\bsv}}$,
the following holds:

With $C_{\mathrm s}:= \max_{i=1}^{N_{\bsv}}\sum_{\ell=2}^{n} |v_{i,\ell}|$ 
bounding the sum of the sizes of coefficients for each polynomial, ignoring the first two coefficients,
and 
$C_{\mathrm n}
	:= |\{ v_{i,\ell} \neq 0: \ell = 0,\ldots,n \text{ and } i=1,\ldots,{N_{\bsv}}\}| 
	\leq (n+1){N_{\bsv}}$ 
bounding the total number of nonzero coefficients of $(v_i)_{i=1}^{N_{\bsv}}$,
there exist ReLU NNs $\{\Phi^{\bsv}_{\tau}\}_{{\tau}\in(0,1)}$ 
with input dimension one and output dimension ${N_{\bsv}}$ 
which satisfy for emulation tolerance $\tau \in (0,1)$
\begin{align*}
\normc[W^{1,\infty}(\hI)]{v_i-\realiz(\Phi^{\bsv}_{\tau})_i}
	\leq &\, {\tau} C_{\mathrm s},
	\qquad \forall i=1,\ldots,{N_{\bsv}},
	\\
\depth(\Phi^{\bsv}_{\tau})
	\leq &\, C_{\depth}(1+\log_2(n))\log_2({1/{\tau}})
		+ \tfrac{2}{3} C_{\depth} (\log_2(n))^3 + C (1+\log_2(n))^2,
	\\
\size(\Phi^{\bsv}_{\tau})
	\leq &\, 4C_{\size} n \log_2({1/{\tau}}) + 16C_{\size} n\log_2(n)
	\\
	&\, + 6C_{\depth} (1+\log_2(n))^2 \log_2({1/{\tau}}) + C (n+C_{\mathrm n}),
	\\
\sizefirst(\Phi^{\bsv}_{\tau})
	\leq &\, \begin{cases} C_{\mathrm n} &\text{ if } n=1,\\ 
		 4(\log_2(n)+1) &\text{ else,} \end{cases}
	\\
\sizelast(\Phi^{\bsv}_{\tau}) \leq &\, 2C_{\mathrm n}.
\end{align*}
In addition, 
\[
\realiz(\Phi^{\bsv}_{\tau})_i(\pm1) = v_i(\pm1)\;, \quad i=1,\ldots,{N_{\bsv}} \;.
\]
The weights and biases in the hidden layers are independent of $(v_i)_{i=1}^{N_{\bsv}}$.
The weights and biases in the output layer are the \Cheb coefficients 
$\{v_{i,\ell} : \ell=0,\ldots,n \text{ and } i=1,\ldots,N_{\bsv}\}$,
which are linear combinations of function values in the \CC points,
as shown in Proposition \ref{prop:cconed} Item \eqref{item:ccifftoned}.
\end{proposition}
We remark that the emulation bounds in Proposition~\ref{prop:relupolynom} 
are better than, e.g. the results in \cite[Proposition 2.3]{MontYangDu21}.
There, the bounds on the network size depend quadratically on the polynomial degree.

The proof of Proposition \ref{prop:relupolynom} is 
based on Lemma \ref{lem:relupolynominduction} below.
It proceeds along the lines and uses ideas in the proof of \cite[Proposition 4.2]{OPS2020}.
As in \cite[Lemma 4.5]{OPS2020}, 
we compute the basis polynomials (in this case \Cheb polynomials) 
using a binary tree of product networks
introduced in Proposition \ref{prop:multiplicationThing}.
The main difference with \cite{OPS2020}
is that we use the recurrence relation \eqref{eq:chebrecur}.
We use it for $|m-n|\in\{0,1\}$, 
such that only one product network is required to compute 
$T_{m+n}(x)$ from $T_{m}(x)$, $T_{n}(x)$ and $T_{|m-n|}(x) \in \{ 1, x \}$. 

The networks constructed in the proofs of 
Lemma \ref{lem:relupolynominduction} and Proposition \ref{prop:relupolynom}
were already defined in \cite[Appendix A]{HOS2022}.
Hitherto, their depth, size and approximation error had not yet been analyzed.

For $k \in \N$, 
the following lemma describes a binary tree-structured ReLU NN 
with $2^{k-1}+2$ outputs.
The first output equals the NN input $x$, 
the other outputs emulate high-order \Cheb polynomials of degrees $2^{k-1}, \ldots, 2^{k}$
to prescribed accuracy $\delta \in (0,1)$.
We retain to the extent possible notation from \cite{OPS2020}.
\begin{lemma}[{\cite[Lemma~6.2.3]{JOdiss}}]
\label{lem:relupolynominduction}
Let $\hI = [-1,1]$.
For all $k\in\N$ there exist NNs $\{\Psi^{k}_{\delta}\}_{\delta\in(0,1)}$ 
with input dimension one and output dimension $2^{k-1}+2$ such that,
denoting all components of the output, except for the first one, by
$\widetilde{T}_{\ell,\delta} := \realiz(\Psi^{k}_{\delta})_{2+\ell-2^{k-1}}$
for $\ell\in\{2^{k-1},\ldots,2^{k}\}$, it holds that 
\begin{align}
\nonumber
\realiz(\Psi^{k}_{\delta})(x) 
	= &\, \big( x, \widetilde{T}_{2^{k-1},\delta}(x), \ldots, \widetilde{T}_{2^k,\delta}(x) \big),
	\hspace{1em} x \in \hI,
	\\	
\normc[W^{1,\infty}(\hI)]{T_{\ell}-\widetilde{T}_{\ell,\delta}}
	\leq &\, \delta,
	\hspace{1em} \ell \in \{2^{k-1},\ldots,2^k\},
	\label{eq:induction_monomerror}
	\\	
\widetilde{T}_{\ell,\delta}(\pm1) = &\, T_{\ell}(\pm1) = (\pm1)^\ell,
	\hspace{1em} \ell \in \{2^{k-1},\ldots,2^k\},
	\label{eq:induction_bdryvalues}
	\\
\depth(\Psi^{k}_{\delta})
	\leq& \, 
	C_{\depth} \big( \tfrac{2}{3}k^3 + 3k^2 + k\log_2(1/\delta) \big)
		+ (5C_{\depth}+C'_{\depth}+2) k,
	\label{eq:induction_depth}
	\\
\size(\Psi^{k}_{\delta})
	\leq &\, 4C_{\size} k 2^k + C_{\size} 2^k \log_2(1/\delta)
		+ 4kC_{\depth} \log_2(1/\delta)
	\nonumber
	\\
	&\, + C_1 2^k + \tfrac{8}{3} C_{\depth} k^3 + 12 C_{\depth} k^2 + C_2 k,
	\label{eq:induction_size}
	\\
\sizefirst(\Psi^{k}_{\delta})
	\leq &\, C_{\first} + 4,
	\label{eq:induction_sizefirst}
	\\
\sizelast(\Psi^{k}_{\delta})
	\leq &\, 2^{k+1}+C_{\last}+4
	\label{eq:induction_sizelast}
.
\end{align}
Here,
$C_1 := 9C_{\size} + C'_{\size} + C_{\first} + C_{\last} + 14$,
$C_2 := 20 C_{\depth} + 4C'_{\depth} + C_{\last} + 24$.
\end{lemma}
\begin{proof}
This proof closely follows that of \cite[Lemma 4.5]{OPS2020},
proving the lemma by induction over $k\in \N$.

\emph{Induction basis.} 
We first prove the result for $k=1$
for all $\delta\in(0,1)$.
Let $L_1 := \depth(\widetilde{\times}_{\delta/4,1})$
and define the matrix $A := [1,1]^\top \in\R^{2\times 1}$,
the vector ${b} := [-1]\in\R^1$,
and the matrices and vectors $A_i$, $b_i$, $i=1,\ldots,L_1$ such that
$\widetilde{\times}_{\delta/4,1} =: ((A_1,b_1),\ldots,(A_{L_1},b_{L_1}))$ 
as in Proposition \ref{prop:multiplicationThing}.
We then define 
$$
\Psi^{1}_{\delta} := 
\Parallel{ \Phi^{\Id}_{1,L_1} , \Phi^{\Id}_{1,L_1} , 
	((A_1 A,b_1),\ldots,(2A_{L_1},2b_{L_1}+{b})) }.
$$
Its realizes $[\realiz(\Psi^{1}_{\delta})(x)]_{1} = x$, 
$\widetilde{T}_{1,\delta}(x) 
	:= [\realiz(\Psi^{1}_{\delta})(x)]_{2} = x = T_{1}(x)$ and 
$\widetilde{T}_{2,\delta}(x) 
	:= [\realiz(\Psi^{1}_{\delta})(x)]_{3} 
	= 2\realiz(\widetilde{\times}_{\delta/4,1})(x,x)-1$
for all $x\in\hI$.
In particular, \eqref{eq:induction_bdryvalues} follows from \eqref{eq:multone}.

For the depth and the size of $\Psi^{1}_{\delta}$ we obtain
\begin{align*}
\depth(\Psi^{1}_{\delta})
	= &\, L_1
		\leq C_{\depth}\log_2(4/\delta)+C'_{\depth},
	\\
\size(\Psi^{1}_{\delta})
	= &\, 2\size(\Phi^{\Id}_{1,L_1}) 
		+ \size(((A_1 A,b_1),\ldots,(2A_{L_1},2b_{L_1}+{b}))) 
	\\
	\leq &\, 4L_1 + \left(C_{\size}\log_2(4/\delta)+C'_{\size}+1\right) 
	\\
	\leq &\, \big(4C_{\depth} + C_{\size}\big) \log_2(4/\delta) 
		+ 4C'_{\depth} + C'_{\size}+1,
	\\
\sizefirst(\Psi^{1}_{\delta})
	= &\, 2\sizefirst(\Phi^{\Id}_{1,L_1}) 
		+ \sizefirst(((A_1 A,b_1),\ldots,(2A_{L_1},2b_{L_1}+{b}))) 
	\\
	\leq &\, C_{\first} + 4,
	\\
\sizelast(\Psi^{1}_{\delta})
	= &\, 2\sizelast(\Phi^{\Id}_{1,L_1}) 
		+ \sizelast(((A_1 A,b_1),\ldots,(2A_{L_1},2b_{L_1}+{b}))) 
	\\
	\leq &\, C_{\last} + 5.
\end{align*} 
Note that only the bound on $\sizelast(\Psi^{1}_{\delta})$ 
gives the term $C_{\last}$ in Equation \eqref{eq:induction_sizelast}.
As we will see below, 
there is no term $C_{\last}$ in the bound on $\sizelast(\Psi^{k}_{\delta})$ for $k>1$.

With Proposition \ref{prop:multiplicationThing} the error can be bounded as
\begin{align*}
\snormc[W^{1,\infty}(\hI)]{(2x^2-1)-\widetilde{T}_{2,\delta}(x)}
	= &\, \normc[L^\infty(\hI)]{4x 
		- 2[D\widetilde{\times}_{\delta/4,1}]_1(x,x) 
		- 2[D\widetilde{\times}_{\delta/4,1}]_2(x,x)}
	\\
	\leq &\, 2\tfrac{\delta}{4} + 2\tfrac{\delta}{4}
	= \delta,
	\\
\normc[W^{1,\infty}(\hI)]{T_{2}-\widetilde{T}_{2,\delta}}
	\leq &\, \delta
,
\end{align*}
where $[D\widetilde{\times}_{\delta/4,1}]$ is the Jacobian, 
which is a $1\times 2$-matrix and is treated as a row vector.
The first inequality, Poincar\'e's inequality and $\widetilde{T}_{2,\delta}(0)=-1=T_{2}(0)$ 
(which follows from Equation \eqref{eq:ZeroInZeroOutBasic})
imply the second inequality.
This shows Equation \eqref{eq:induction_monomerror} for $k=1$
and finishes the proof of the induction basis.

\emph{Induction hypothesis (IH).} 
For all $\delta\in(0,1)$ and $k\in\N$ 
we define $\theta := 2^{-2k-4} \delta$
and assume that there exists a NN $\Psi^{k}_{\theta}$ 
which satisfies \eqref{eq:induction_monomerror}--\eqref{eq:induction_sizelast} 
with $\theta$ instead of $\delta$. 

\emph{Induction step.}
For $\delta$ and $k$ as in (IH), we show that
\eqref{eq:induction_monomerror}--\eqref{eq:induction_sizelast} 
hold with $\delta$ as in (IH) and with $k+1$ instead of $k$. 

For all $\ell\in\{2^{k-1},\ldots,2^k\}$,
\begin{align}
\normc[L^\infty(\hI)]{\widetilde{T}_{\ell,\theta}}
	\leq &\, \normc[L^\infty(\hI)]{T_{\ell}}
		+ \normc[W^{1,\infty}(\hI)]{T_{\ell}-\widetilde{T}_{\ell,\theta}(x)}
	\leq 1 + \theta < 2,
	\label{eq:xldeltaLinfty}
\end{align}
so that we may use $\widetilde{T}_{\ell,\theta}(x)$ as input of 
$\widetilde{\times}_{\theta,2}$.

We define,
for $\Phi^{1,k}$ and $\Phi^{2,k}_{\delta}$ introduced below, 
\begin{align}
\label{eq:defPsik+1}
\Psi^{k+1}_{\delta} := 
	\Phi^{2,k}_{\delta} \odot \Phi^{1,k} \odot \Psi^{k}_{\theta}.
\end{align}
The NN $\Phi^{1,k}$ implements the linear map
\begin{align*}
\R^{2^{k-1}+2}\to\R^{2^{k+1}+2}:\Big(z_1,\ldots,z_{2^{k-1}+2})
 \mapsto
(z_1,z_{2^{k-1}+2},z_{2},z_{3},z_3,z_3,z_3,z_4,z_4,z_4,z_4,z_5,
\\
\ldots,z_{2^{k-1}+1},z_{2^{k-1}+2},z_{2^{k-1}+2},z_{2^{k-1}+2} \Big).
\end{align*}
Denoting its weights by $((A^{1,k},b^{1,k})):= \Phi^{1,k}$,
it holds that $b^{1,k} = 0$ and 
\begin{align*}
(A^{1,k})_{m,i} = 
	\begin{cases}
		1 & \text{if } m = 1, i = 1,
		\\
		1 & \text{if } m = 2, i = 2^{k-1} + 2,
		\\
		1 & \text{if } m \in \{3,\ldots,2^{k+1}+2\}, 
			i = \lceil \tfrac{m+5}{4} \rceil,
		\\
		0 & \text{else}.
	\end{cases}
\end{align*}
It follows that
\begin{align*}
\depth(\Phi^{1,k}) 
	= 1,
	\qquad 
\sizefirst(\Phi^{1,k}) = \sizelast(\Phi^{1,k}) = \size(\Phi^{1,k})  
	\leq 2^{k+1}+2
.
\end{align*}
With $L_{\theta} := \depth(\widetilde{\times}_{\theta,2})$ we define 
\begin{align*}
\Phi^{2,k}_{\delta} 
	:= ((A^{2,k},b^{2,k}))
	\sconc \dParallel{ \Phi^{\Id}_{2,L_{\theta}} , 
	\widetilde{\times}_{\theta,2} , \ldots , \widetilde{\times}_{\theta,2} },
\end{align*}
containing $2^{k}$ $\widetilde{\times}_{\theta,2}$-networks,
with $A^{2,k}\in\R^{(2^k+2)\times (2^k+2)}$ and $b^{2,k}\in\R^{2^k+2}$ 
defined as
\begin{align*}
(A^{2,k})_{m,i} := \begin{cases}
	1 & \text{if } m = i \leq 2, \\
	2 & \text{if } m = i \geq 3, \\
	-1 & \text{if } m\geq3 \text{ is odd}, i=1, \\ 
	0 & \text{else},
	\end{cases}
\qquad
(b^{2,k})_m = \begin{cases} 
	-1 & \text{if } m\geq 3 \text{ is even}, \\ 
	0 & \text{else}.
	\end{cases}
\end{align*}
The network depth and size of $\Phi^{2,k}_{\delta}$ can be estimated as follows: 
$\depth((A^{2,k},b^{2,k})) = 1$, 
$\size((A^{2,k},b^{2,k})) = \sizefirst((A^{2,k},b^{2,k})) = \sizelast((A^{2,k},b^{2,k}))
	= 2\cdot2^k+2$
and 
\begin{align*}
\depth(\Phi^{2,k}_{\delta})
	= &\, \depth((A^{2,k},b^{2,k})) + \depth(\widetilde{\times}_{\theta,2})
		\leq 1+ C_{\depth} \big( \log_2(2/\theta) \big) + C'_{\depth}\\
	= &\, C_{\depth} \big( 2k+5+\log_2(1/\delta) \big) + C'_{\depth} +1,
	\\
\size(\Phi^{2,k}_{\delta})
	\leq &\, \size((A^{2,k},b^{2,k})) + \sizefirst((A^{2,k},b^{2,k}))
		+ \sizelast(\Phi^{\Id}_{2,L_{\theta}}) 
		+ 2^{k} \sizelast(\widetilde{\times}_{\theta,2})
		+ \size(\Phi^{\Id}_{2,L_{\theta}}) 
	\\
	&\, + 2^{k} \size(\widetilde{\times}_{\theta,2})
	\\
	\leq &\, (2^{k+1} +2) + (2^{k+1} +2) + 4 + C_{\last} 2^k 
		+ 4 \depth(\widetilde{\times}_{\theta,2}) 
		+ 2^{k} \size(\widetilde{\times}_{\theta,2})
	\\
	\leq &\, (4C_{\depth} + C_{\size} 2^k) \log_2(2/\theta)
		 + (4 + C_{\last} + C'_{\size}) 2^k + 4C'_{\depth} + 8
	\\
	\leq &\, (4C_{\depth} + C_{\size} 2^k) (2k + 5 + \log_2(1/\delta))
		 + (4 + C_{\last} + C'_{\size}) 2^k + 4C'_{\depth} + 8
	\\
\sizefirst(\Phi^{2,k}_{\delta})
	= &\, \sizefirst(\Phi^{\Id}_{2,L_{\theta}}) 
		+ 2^k \sizefirst(\widetilde{\times}_{\theta,2})
	\leq C_{\first} 2^k + 4,
	\\
\sizelast(\Phi^{2,k}_{\delta})
	= &\, 2 \sizelast((A^{2,k},b^{2,k}))
	\leq 2^{k+2}+4
.
\end{align*}

The network $\Psi^{k+1}_{\delta}$
defined in Equation \eqref{eq:defPsik+1} realizes 
\begin{align}
[\realiz(\Psi^{k+1}_{\delta})(x)]_{1}
	&= x,
	\qquad
	\text{ for }
	x\in\hI,
	\label{eq:psik+1R1}
	\\
[\realiz(\Psi^{k+1}_{\delta})(x)]_{2}
	&= \widetilde{T}_{2^k,\theta}(x),
	\qquad
	\text{ for }
	x\in\hI,
	\label{eq:psik+1R2}
	\\
[\realiz(\Psi^{k+1}_{\delta})(x)]_{\ell+2-2^k}
	&= 2\realiz(\widetilde{\times}_{\theta,2})\left(\widetilde{T}_{\lceil \ell/2 \rceil,\theta}(x), 
		\widetilde{T}_{\lfloor \ell/2 \rfloor,\theta}(x)\right) 
		- x^{\lceil \ell/2 \rceil - \lfloor \ell/2 \rfloor},
	\label{eq:psik+1Rell+1-2^k}
	\\\nonumber
	&\qquad
	\text{ for }
	x\in\hI
	\text{ and }
	\ell \in\{2^k+1,\ldots,2^{k+1}\},
\end{align}
where $x^{\lceil \ell/2 \rceil - \lfloor \ell/2 \rfloor}=x=T_{1}(x)$ if $\ell$ is odd
and $x^{\lceil \ell/2 \rceil - \lfloor \ell/2 \rfloor}=1=T_{0}(x)$ if $\ell$ is even.
For $\ell \in\{2^k+1,\ldots,2^{k+1}\}$ and $x\in\hI$ 
we denote the right-hand side of \eqref{eq:psik+1Rell+1-2^k} by
\begin{align*}
\widetilde{T}_{\ell,\delta}(x)
	:= &\, [\realiz(\Psi^{k+1}_{\delta})(x)]_{\ell+2-2^k}.
\end{align*}

We bound the depth and the size of $\Psi^{k+1}_{\delta}$.
\begin{align*}
\depth(\Psi^{k+1}_{\delta})
	= &\, \depth(\Phi^{2,k}_{\delta}) + \depth(\Phi^{1,k}) 
		+ \depth(\Psi^{k}_{\theta})
	\\
	\leq &\, \left( C_{\depth} \big( 2k+5+\log_2(1/\delta) \big) 
		+ C'_{\depth} +1 \right) + 1
	\\
	&\, + \left(C_{\depth} \big(\tfrac{2}{3}k^3 + 3k^2 
		+ k \log_2(2^{2k+4}/\delta) \big) 
		+ (5 C_{\depth} + C'_{\depth} + 2) k \right)
	\\
	\leq &\, C_{\depth}\big( \tfrac{2}{3}(k+1)^3 + 3(k+1)^2 
		+ (k+1)\log_2(1/\delta) \big) + (5C_{\depth} + C'_{\depth}+2) (k+1),
	\\
\size(\Psi^{k+1}_{\delta})
	\leq &\, \size(\Phi^{2,k}_{\delta}) + \sizefirst(\Phi^{2,k}_{\delta})
		+ \sizelast(\Phi^{1,k}\odot \Psi^{k}_{\theta})
		+ \size(\Phi^{1,k}\odot\Psi^{k}_{\theta})
	\\
	\leq &\, \size(\Phi^{2,k}_{\delta}) + \sizefirst(\Phi^{2,k}_{\delta}) 
		+ 2\sizelast(\Phi^{1,k}) + \size(\Phi^{1,k}) + \sizefirst(\Phi^{1,k})
		+ \sizelast(\Psi^{k}_{\theta}) 
	\\
	&\, + \size(\Psi^{k}_{\theta})
	\\
	\leq &\, \left( (4C_{\depth} + C_{\size} 2^k) (2k + 5 + \log_2(1/\delta))
		 + (4 + C_{\last} + C'_{\size}) 2^k + 4C'_{\depth} + 8 \right) 
	\\
	&\, + (C_{\first}2^k+4)
		+ 2 (2^{k+1}+2) 
		+ (2^{k+1}+2)
		+ (2^{k+1}+2) 
		+ (2^{k+1} + C_{\last} + 4) 
	\\
	&\, + \Big( 4C_{\size} k 2^k + C_{\size} 2^k \log_2(2^{2k+4}/\delta)
		+ 4kC_{\depth}\log_2(2^{2k+4}/\delta)
		+ C_1 2^k 
	\\
	&\, + \tfrac{8}{3}C_{\depth} k^3  + 12k^2C_{\depth}  + C_2 k \Big)
	\\
	\leq &\, 4C_{\size} (k+1) 2^{k+1} + C_{\size} 2^{k+1} \log_2(1/\delta)
		+ 4(k+1) C_{\depth} \log_2(1/\delta) 
	\\
	&\, + C_1 2^{k+1} + \tfrac{8}{3} C_{\depth} (k+1)^3 + 12C_{\depth}(k+1)^2  
		+ C_2 (k+1),
	\\
C_1 := &\, 9C_{\size} + C'_{\size} + C_{\first} + C_{\last} + 14,
	\\
C_2 := &\, 20 C_{\depth} + 4C'_{\depth} + C_{\last} + 24,
	\\
\sizefirst(\Psi^{k+1}_{\delta})
	= &\, \sizefirst(\Psi^{k}_{\theta}) 
		\leq C_{\first} + 4,
	\\
\sizelast(\Psi^{k+1}_{\delta})
	= &\, \sizelast(\Phi^{2,k}_{\delta}) \leq 2^{(k+1)+1} +4
.
\end{align*}
This proves
Equations \eqref{eq:induction_depth}--\eqref{eq:induction_sizelast} for $k+1$.

It remains to show Equations \eqref{eq:induction_monomerror}--\eqref{eq:induction_bdryvalues}.
For $\ell=2^{k}$ Equation \eqref{eq:induction_monomerror} 
follows from $\theta<\delta$ and Equation \eqref{eq:psik+1R2}.
Towards a proof for $\ell\in\{2^k+1,\ldots,2^{k+1}\}$,
we note that $\normc[L^\infty(\hI)]{T_{m}'} = m^2$ for all $m\in\N$:
the inequality $\normc[L^\infty(\hI)]{T_{m}'} \geq m^2$ follows from
$U_{m-1}(1)=m$ (\cite[Section 1.5.1]{Gautschi2004}) and $T_{m}' = mU_{m-1}$, 
where $U_{m-1}$ is the \Cheb polynomial of the second kind of degree $m-1$.
The converse inequality is Markov's inequality Lemma \ref{lem:markovoned}.
For $\ell\in\{2^k+1,\ldots,2^{k+1}\}$ we write $\ell_0:=\lceil \ell/2\rceil$
and observe that for $m\in\{\ell_0,\ell-\ell_0\}$ 
\begin{align*}
\normc[L^\infty(\hI)]{\widetilde{T}_{m,\theta}}
	\leq &\, 1 + \theta 
	< 2
	,
	\\
\normc[L^\infty(\hI)]{\tfrac{\mathrm{d}}{\mathrm{d} x}\widetilde{T}_{m,\theta}}
	\leq &\, \normc[L^\infty(\hI)]{T_{m}'} 
		+ \normc[W^{1,\infty}(\hI)]{T_{m}-\widetilde{T}_{m,\theta}}
	\leq m^2 + \theta 
	< m^2 + 1
.
\end{align*}
With Equations \eqref{eq:chebrecur} and \eqref{eq:psik+1Rell+1-2^k} we find
(note that the $x^{\lceil \ell/2 \rceil - \lfloor \ell/2 \rfloor}$-terms cancel)
\begin{align*}
\normc[L^\infty(\hI)]{T_{\ell} - \widetilde{T}_{\ell,\delta}}
	\leq &\, \normc[L^\infty(\hI)]{ 2T_{\ell_0}{} \left( T_{\ell-\ell_0}{} 
			- \widetilde{T}_{\ell-\ell_0,\theta}{} \right) }
		+ \normc[L^\infty(\hI)]{ 2\widetilde{T}_{\ell-\ell_0,\theta}{} 
			\left( T_{\ell_0}{} - \widetilde{T}_{\ell_0,\theta}{} \right) }
	\\
	&\, + \normc[L^\infty(\hI)]{ 2 \widetilde{T}_{\ell_0,\theta}{}
			\widetilde{T}_{\ell-\ell_0,\theta}{}
			- 2 \realiz(\widetilde{\times}_{\theta,2}) \big( 
			\widetilde{T}_{\ell_0,\theta}{} , 	
			\widetilde{T}_{\ell-\ell_0,\theta}{} \big) }
	\\
	\leq &\, 2\theta + 4\theta + 2\theta \leq \delta
	,
	\\
\snormc[W^{1,\infty}(\hI)]{T_{\ell}-\widetilde{T}_{\ell,\delta}}
	\leq &\, \normc[L^\infty(\hI)]{ 2 T_{\ell-\ell_0}{} \tfrac{\mathrm{d}}{\mathrm{d}x} T_{\ell_0}{}
		- 2[D\realiz(\widetilde{\times}_{\theta,2})]_1 \big( 
			\widetilde{T}_{\ell_0,\theta}{} , 	
			\widetilde{T}_{\ell-\ell_0,\theta}{} \big) 
		\tfrac{\mathrm{d}}{\mathrm{d} x} \widetilde{T}_{\ell_0,\theta}{} }
	\\
	&\, + \normc[L^\infty(\hI)]{ 2 T_{\ell_0}{} \tfrac{\mathrm{d}}{\mathrm{d} x} T_{\ell-\ell_0}{}
		- 2[D\realiz(\widetilde{\times}_{\theta,2})]_2 \big( 
			\widetilde{T}_{\ell_0,\theta}{} , \
			\widetilde{T}_{\ell-\ell_0,\theta}{} \big) 
		\tfrac{\mathrm{d}}{\mathrm{d} x} \widetilde{T}_{\ell-\ell_0,\theta}{} }
	\\
	\leq &\, \normc[L^\infty(\hI)]{ 2 \left( \tfrac{\mathrm{d}}{\mathrm{d}x} T_{\ell_0}{} \right)
		\big( T_{\ell-\ell_0}{} - \widetilde{T}_{\ell-\ell_0,\theta}{}\big) }
	\\
	&\, + \normc[L^\infty(\hI)]{ 2\widetilde{T}_{\ell-\ell_0,\theta}{} 
		\big( \tfrac{\mathrm{d}}{\mathrm{d}x} T_{\ell_0}{}
			- \tfrac{\mathrm{d}}{\mathrm{d}x} \widetilde{T}_{\ell_0,\theta}{} \big) }
	\\
	&\, + \normc[L^\infty(\hI)]{ 2\big( \widetilde{T}_{\ell-\ell_0,\theta}{} - 
		[D\realiz(\widetilde{\times}_{\theta,2})]_1 \big( \widetilde{T}_{\ell_0,\theta}{} , 
			\widetilde{T}_{\ell-\ell_0,\theta}{} \big) \big) 
		\tfrac{\mathrm{d}}{\mathrm{d}x} \widetilde{T}_{\ell_0,\theta}{} }
	\\
	&\, + \normc[L^\infty(\hI)]{ 2 \left( \tfrac{\mathrm{d}}{\mathrm{d}x} T_{\ell-\ell_0}{} \right)
		\big( T_{\ell_0}{} - \widetilde{T}_{\ell_0,\theta}{} \big) }
	\\
	&\, + \normc[L^\infty(\hI)]{ 2\widetilde{T}_{\ell_0,\theta}{} 
		\big( \tfrac{\mathrm{d}}{\mathrm{d}x}T_{\ell-\ell_0}{} 
			- \tfrac{\mathrm{d}}{\mathrm{d} x} \widetilde{T}_{\ell-\ell_0,\theta}{} 
		\big) }
	\\
	&\, + \normc[L^\infty(\hI)]{ 2\big( \widetilde{T}_{\ell_0,\theta}{} 
		- [D\realiz(\widetilde{\times}_{\theta,2})]_2 \big( \widetilde{T}_{\ell_0,\theta}{} , 
			\widetilde{T}_{\ell-\ell_0,\theta}{} \big) \big) 
		\tfrac{\mathrm{d}}{\mathrm{d} x} \widetilde{T}_{\ell-\ell_0,\theta}{} }
	\\
	\stackrel{\eqref{eq:xldeltaLinfty},\text{(IH)}}{\leq} 
	&\, 2(\ell_0)^2 \theta + 4\theta + 2((\ell_0)^2+1) \theta + 2(\ell-\ell_0)^2 \theta 
		+ 4\theta + 2((\ell-\ell_0)^2+1) \theta
	\\
	\leq &\, (4\ell^2 -8\ell_0(\ell-\ell_0)+12)\theta 
	\leq (4\ell^2 -16+12)\theta
	\leq  4\ell^2 \theta 
	\leq \delta,
\end{align*}
where $[D\realiz(\widetilde{\times}_{\delta,2})]$ is the Jacobian 
and where we have used that $3\leq \ell \leq 2^{k+1}$, 
$\ell_0\geq 2$, $\ell-\ell_0\geq1$,
and $4 \ell^2\leq 2^{2k+4}$.
We conclude that
$\norm[W^{1,\infty}(\hI)]{T_{\ell} - \widetilde{T}_{\ell,\delta}}
	\leq \delta$.
Finally, we conclude \eqref{eq:induction_bdryvalues} from \eqref{eq:psik+1Rell+1-2^k} and \eqref{eq:multone}.

This finishes the induction step: 
based on the induction hypothesis for $\delta\in(0,1)$ and $k\in\N$,
we have constructed $\Psi^{k+1}_{\delta}$ and shown 
Equations \eqref{eq:induction_monomerror}--\eqref{eq:induction_sizelast} 
for 
$k+1$ instead of $k$. 
Together with the induction basis, this finishes the proof of the lemma.
\end{proof}
Because the number of concatenations in the inductive construction of $\Psi^{k}_{\delta}$ depends on $k$,
we had to use the sharper bound in Equation \eqref{eq:concsize} 
involving $\sizefirst(\cdot)$ and $\sizelast(\cdot)$.
Using the second inequality in \eqref{eq:concsize} would have introduced an extra $k$-dependent factor 
in the bound on the network size.

\begin{proof}[Proof of Proposition \ref{prop:relupolynom}]
We define $v_{i,\ell}:=0$ for all $\ell>n$ and $i=1,\ldots,{N_{\bsv}}$.

If $n=1$, for all $\tau \in(0,1)$ we define 
$\Phi^{\bsv}_{\tau}:= ((A,b))$
for $A = (A_{i1})_{i=1}^{N_{\bsv}} = (v_{i,1})_{i=1}^{N_{\bsv}}\in\R^{{N_{\bsv}}\times1}$ 
and $b= (b_{i})_{i=1}^{N_{\bsv}} = (v_{i,0})_{i=1}^{N_{\bsv}}\in\R^{N_{\bsv}}$. 
It follows that 
$\normc[W^{1,\infty}(\hI)]{v_i-\realiz(\Phi^{\bsv}_{\tau})_i} = 0$ for all $i\in\{1,\ldots,{N_{\bsv}}\}$,
$\depth(\Phi^{\bsv}_{\tau}) = 1$ and  
$\size(\Phi^{\bsv}_{\tau}) 
	= \sizefirst(\Phi^{\bsv}_{\tau})
	= \sizelast(\Phi^{\bsv}_{\tau})
	\leq C_{\mathrm n}$. 

If $n\geq2$, let 
$k:=\lceil \log_2(n)\rceil$ and $\delta = {\tau}$.
We use Lemma \ref{lem:relupolynominduction} and take $\{\ell_j\}_{j=1}^k\in\N^k$ such that 
$\depth\left( \Psi^k_{\delta} \right)+1 
	= \depth\left( \Psi^j_{\delta} \right) + \ell_j$ for $j=1,\ldots,k$, 
and thus $\ell_j\leq \max_{j=1}^k \depth\left( \Psi^j_{\delta} \right) = \depth\left( \Psi^k_{\delta} \right)$. 
We define 
\begin{align*}
\Phi^{\bsv}_{\tau} := \Phi^{3,n} \sconc 
	\Parallel{\Psi^1_{\delta}\sconc\Phi^{\Id}_{1,\ell_1},\ldots,
		\Psi^k_{\delta}\sconc\Phi^{\Id}_{1,\ell_k}}.
\end{align*}
The NN $\Phi^{3,n}$ emulates the affine map $\R^{2^k+2k-1}\to\R^{N_{\bsv}}$:
\begin{align*}
(z_1,\ldots,z_{2^k+2k-1}) \mapsto 
	\left(v_{i,0} + v_{i,1} z_2 + v_{i,2} z_3
	+ \sum_{j=2}^k\sum_{\ell = 2^{j-1}+1 }^{2^{j}} v_{i,\ell} z_{\ell+2j-1}\right)_{i=1}^{N_{\bsv}}.
\end{align*}
It satisfies $\depth(\Phi^{3,n}) = 1$ and
$\size(\Phi^{3,n}) = \sizefirst(\Phi^{3,n}) = \sizelast(\Phi^{3,n}) 
	\leq C_{\mathrm n}$.

The realization satisfies
\begin{align*}
\realiz(\Phi^{\bsv}_{\tau})_i(x) 
	= v_{i,0} + \sum_{\ell=1}^{2^k} v_{i,\ell} \widetilde{T}_{\ell,\delta}(x),
	\qquad
	x \in \hI, \, i\in\{1,\ldots,{N_{\bsv}}\}.
\end{align*}
Exactness in the points $\pm1$ follows from 
Equation \eqref{eq:induction_bdryvalues} in Lemma \ref{lem:relupolynominduction}.

The depth and the size of $\Phi^{\bsv}_{\tau}$ can be bounded, using that $2^k\leq 2n$:
\begin{align*}
\depth\left(\Phi^{\bsv}_{\tau}\right)=&\,\depth\left(\Phi^{3,n}\right) 
	+ \left( \depth\left(\Psi^{k}_{\delta}\right) + 1\right)
	\\
	\leq &\, 2 + \left(C_{\depth} \left(\tfrac{2}{3}k^3 
		+ k\log_2\left({1/{\tau}}\right) \right) + C k^2 \right)
	\\
	\leq &\, C_{\depth}(1+\log_2(n))\log_2\left({1/{\tau}}\right)
		+ \tfrac{2}{3}C_{\depth}\log_2^3(n) + C\log_2^2(n),
	\\
\size\left(\Phi^{\bsv}_{\tau}\right) 
	\leq &\, \size\left(\Phi^{3,n}\right) + \sizefirst\left(\Phi^{3,n}\right)
		+ \sum_{j=1}^k \sizelast\left(\Psi^{j}_{\delta}\sconc\Phi^{\Id}_{1,\ell_j}\right) 
		+ \sum_{j=1}^k \size\left(\Psi^{j}_{\delta}\sconc\Phi^{\Id}_{1,\ell_j}\right) 
	\\
	\leq &\, \size\left(\Phi^{3,n}\right) + \sizefirst\left(\Phi^{3,n}\right)
		+ \sum_{j=1}^k \sizelast\left(\Psi^{j}_{\delta}\right) 
		+ \sum_{j=1}^k \size\left(\Psi^{j}_{\delta}\right) 
		+ \sum_{j=1}^k \sizefirst\left(\Psi^{j}_{\delta}\right) 
	\\
	&\, + \sum_{j=1}^k \sizelast\left(\Phi^{\Id}_{1,\ell_j}\right) 
		+ \sum_{j=1}^k \size\left(\Phi^{\Id}_{1,\ell_j}\right) 
	\\
	\leq &\, C_{\mathrm n} + C_{\mathrm n}
		+ \sum_{j=1}^k \left( 2^{j+1}+C_{\last}+4 \right)
	\\
	&\, 
		+ \sum_{j=1}^k \left( 4C_{\size} j 2^j 
		+ C_{\size} 2^j \log_2\left({1/{\tau}}\right)
		+ 4jC_{\depth} \log_2\left({1/{\tau}}\right) + C 2^j \right)
	\\
	&\, 
		+ \left(C_{\first}+4\right) k
		+ 2k
		+ 2k \left(C_{\depth} \left(\tfrac{2}{3}k^3 
		+ k\log_2\left({1/{\tau}}\right) \right) + C k^2 \right)
	\\
	\leq &\, 4C_{\size} n \log_2\left({1/{\tau}}\right) 
		+ 16C_{\size} n\log_2(n)
		+ 6C_{\depth} (1+\log_2(n))^{2} \log_2\left({1/{\tau}}\right) 
	\\
	&\, + C(n + C_{\mathrm n}),
	\\
\sizefirst\left(\Phi^{\bsv}_{\tau}\right)
	= &\, \sum_{j=1}^k \sizefirst\left(\Psi^{j}_{\delta}\sconc\Phi^{\Id}_{1,\ell_j}\right)
	\leq \sum_{j=1}^k 2\sizefirst\left(\Phi^{\Id}_{1,\ell_j}\right)
	= 4k \leq 4(\log_2(n)+1),
\\
\sizelast\left(\Phi^{\bsv}_{\tau}\right)
	= &\, 2\sizelast\left(\Phi^{3,n}\right)
	\leq 2C_{\mathrm n}.
\end{align*}

With the error bounds from Lemma \ref{lem:relupolynominduction} we obtain that
for all $i\in\{1,\ldots,{N_{\bsv}}\}$
\begin{align*}
\normc[W^{1,\infty}(\hI)]{v_i-\realiz\left(\Phi^{\bsv}_{\tau}\right)_i}
	\leq &\, \sum_{\ell=1}^{n} |v_{i,\ell}| 
			\normc[W^{1,\infty}(\hI)]{T_{\ell}-\widetilde{T}_{\ell,\delta}}
	\leq \sum_{\ell=2}^{n} |v_{i,\ell}| \delta
	\leq {\tau} C_{\mathrm s}.
\end{align*}

Finally, the fact that the hidden layer weights and biases of $\Phi^{\bsv}_{\tau}$
are independent of $(v_i)_{i=1}^{N_{\bsv}}$
can be seen directly from its definition. 
For each $i=1,\ldots,N_{\bsv}$,
the Lagrange interpolant of $v_i$ of degree $n$ in the \CC points $( \hat{x}^{\cc,n}_{j} )_{j=0}^n$
equals $v_i$
by Proposition \ref{prop:cconed} Item \eqref{item:ccprojoned},
which means that \ref{prop:cconed} Item \eqref{item:ccifftoned} can be used
to compute the \Cheb coefficients of $v_i$.
\end{proof}

Preparing for the approximation of univariate, piecewise polynomial functions, 
we next derive as a corollary from Proposition \ref{prop:relupolynom}
a NN approximation of multiple 
polynomials on a bounded interval $I:=[a,b]$ for arbitrary $-\infty<a<b<\infty$, 
such that on $(-\infty,a)$ the NN realization 
exactly equals the values of the polynomials in the left endpoint $a$,
and on $(b,\infty)$ equals the values of the polynomials in the right endpoint $b$.
Also, we use the material from Section \ref{sec:prelim},
in particular Lemma \ref{lem:chebclinftyoned} bounding \Cheb coefficients
and the inverse inequalities from Section \ref{sec:inverseineq},
to transform the right-hand side of the error bound in Proposition \ref{prop:relupolynom}.
It is stated in terms of \Cheb coefficients, and we rewrite it
into a bound in terms of Sobolev norms of the approximated polynomial,
taking into account the natural scaling of such norms 
in terms of the interval length $h$.

In the statement of the corollary,
the minimum over $v$ in Equation \eqref{eq:relupolynomcorerr}
expresses the fact that the error is not affected by adding a polynomial of degree $1$ 
to the polynomials we want to approximate,
because NNs with ReLU activation can emulate 
continuous, piecewise linear functions exactly.
\begin{corollary}[{\cite[Corollary 6.2.4]{JOdiss}}]
\label{cor:relupolynomH10}
There exists a constant $C>0$ such that for all $n,{N_{\bsv}}\in\N$ 
and ${N_{\bsv}}$ polynomials 
$v_i = \sum_{\ell=0}^n v_{i,\ell} T_{\ell} \in \bbP_n$ for $i=1,\ldots,{N_{\bsv}}$,
the following holds:

For $-\infty<a<b<\infty$ let $I:=[a,b]$ and denote $h=b-a$.
There exist ReLU NNs $\{\Phi^{\bsv,I}_{\eps}\}_{\eps\in(0,1)}$ 
with input dimension one and output dimension ${N_{\bsv}}$ 
which satisfy for all $i=1,\ldots,{N_{\bsv}}$, 
$1\leq r, r' \leq\infty$ and $t=0,1$
\begin{align}
\realiz(\Phi^{\bsv,I}_{\eps})_i |_{(-\infty,a]}
	= &\, v_i(a),
	\label{eq:relupolynomcorleft}
	\\
\realiz(\Phi^{\bsv,I}_{\eps})_i |_{[b,\infty)}
	= &\, v_i(b),
	\label{eq:relupolynomcorright}
	\\
(2/h)^{1-t} \snormc[W^{t,r}(I)]{v_i-\realiz(\Phi^{\bsv,I}_{\eps})_i}
	\leq &\, \tfrac12 \eps (2/h)^{1+1/r'-1/r} 
		\min_{v\in\bbP_n: \atop v'' = {v_i}''} \normc[L^{r'}(I)]{v}
	\nonumber\\
	\leq &\, \eps (2/h)^{1/r'-1/r} \snormc[W^{1,r'}(I)]{v_i}	
	,
	\label{eq:relupolynomcorerr}
\end{align}
\begin{align}
\nonumber
\depth(\Phi^{\bsv,I}_{\eps})
	\leq &\, C_{\depth}(1+\log_2(n)) \log_2\big({1}/\eps\big) 
		+ \tfrac{2}{3} C_{\depth} (\log_2(n))^3 + C (1+\log_2(n))^2
		,
	\\\nonumber
\size(\Phi^{\bsv,I}_{\eps})
	\leq &\, 4C_{\size} n \log_2\big({1}/\eps\big) + 40C_{\size} n\log_2(n)
		\\\nonumber
	&\, 
		+ 6C_{\depth} (1+\log_2(n))^2 \log_2\big({1}/\eps\big) 
		+ C n {N_{\bsv}}
	,
	\\\nonumber
\sizefirst(\Phi^{\bsv,I}_{\eps})
	\leq &\, 4,
	\\\nonumber
\sizelast(\Phi^{\bsv,I}_{\eps})
	\leq &\, 2(n+1){N_{\bsv}}.
\end{align}
The weights and biases in the hidden layers are independent of $(v_i)_{i=1}^{N_{\bsv}}$.
The weights and biases in the output layer 
are the \Cheb coefficients,
which are linear combinations of the function values in the \CC points in $I$.
\end{corollary}

\begin{proof}
The proof consists of 2 steps.
In Step 1, we introduce a piecewise linear map which enables us to use
Proposition \ref{prop:relupolynom} on the reference interval $[-1,1]$.
We define it such that
Equations \eqref{eq:relupolynomcorleft} and \eqref{eq:relupolynomcorright}
will be satisfied.
In Step 2, we define the network $\Phi^{\bsv,I}_{\eps}$ 
and estimate its error, depth and size.

\textbf{Step 1.}
Let $P: \R\to[-1,1]: x\mapsto \min\{ \max\{ \tfrac{2}{b-a}(x-\tfrac{a+b}{2}), -1 \}, 1 \}$,
which is the continuous, piecewise linear function which is affine on $[a,b]$
and satisfies $P(x) = -1$ for all $x\in(-\infty,a]$ and $P(x) = 1$ for all $x\in[b,\infty)$.
We denote the inverse of its restriction to $[a,b]$ 
by $P^{-1}: [-1,1]\to[a,b]: x\mapsto \tfrac{b-a}{2}x+\tfrac{a+b}{2}$.

For all $i\in\{1,\ldots,{N_{\bsv}}\}$, define $\hat{v}_i := v_i \circ P^{-1}$ on $[-1,1]$.
We denote the \Cheb expansion of $\hat{v}_i$ by $\hat{v}_i=\sum_{\ell=0}^{n} \hat{v}_{i,\ell} T_{\ell}$,
defining $\hat{v}_{i,\ell} = 0$ if $\ell>n$. 
With 
Lemma \ref{lem:chebclinftyoned}
it follows that 
\begin{align*}
\sum_{\ell\geq 2} |\hat{v}_{i,\ell}| 
	\leq n^4 \min_{\hat{v}\in\bbP_n: \atop \hat{v}'' = {\hat{v}_i}''} \normc[L^\infty(\hI)]{\hat{v}}
	.
\end{align*}
With the inverse inequality \eqref{eq:invineqoned}
it follows that for all $1\leq r'<\infty$ and
$\hat{v}\in\bbP_n$, 
with $v := \hat{v}\circ P$:
\begin{align}
\label{eq:inverseLinftyLp}
\normc[L^\infty(\hI)]{\hat{v}}
	\leq &\, \left((r'+1)n^2\right)^{1/r'} \normc[L^{r'}(\hI)]{\hat{v}}
	\leq C_0 n^2 \normc[L^{r'}(\hI)]{\hat{v}}
	\leq C_0 n^2 (2/h)^{1/r'} \normc[L^{r'}(I)]{v}
	,
\end{align}
for an absolute constant $C_0$ independent of $n$, $r'$, $I$ and $\hat{v}$. 
We used that the function $(x+1)^{1/x} = \exp(\tfrac{1}{x}\log(1+x))$ equals $2$ in $x=1$ 
and converges to $1$ for $x\to\infty$, hence by continuity it is bounded on $[1,\infty)$. 
It gives
\begin{align*}
\sum_{\ell\geq 2} |\hat{v}_{i,\ell}| 
	\leq &\, C_0 n^6 (2/h)^{1/r'} \min_{v\in\bbP_n: \atop v'' = {v_i}''} \normc[L^{r'}(I)]{v}.
\end{align*}

\textbf{Step 2.}
In this step, we define the network $\Phi^{\bsv,I}_{\eps}$.

Let $\Phi^P = (([1,1]^\top,[-a,-b]^\top),([2/h,-2/h],[-1]))$ 
be a NN of depth $2$ and size at most $7$ which exactly emulates $P$.
Let $\hat{\bsv} = (\hat{v}_i)_{i=1}^{N_{\bsv}}$ be as defined in Step 1,
and let $\Phi^{\hat{\bsv}}_{\tau}$ for ${\tau} = \eps/(4C_0 n^6)$
be the network from Proposition \ref{prop:relupolynom}.
Then, we define 
\begin{align*}
\Phi^{\bsv,I}_{\eps} := \Phi^{\hat{\bsv}}_{\tau} \sconc \Phi^P
.
\end{align*}
Equations \eqref{eq:relupolynomcorleft} and \eqref{eq:relupolynomcorright}
follow from the definition of $P$ and exactness in $\pm1$ 
of the network from Proposition \ref{prop:relupolynom}.
By the result of Step 1 of this proof, 
for $t = 0,1$ 
the error is bounded as follows:
\begin{align*}
&\,
(\tfrac2h)^{1-t}
	\snormc[W^{t,\infty}(I)]{v_i-\realiz(\Phi^{\bsv,I}_{\eps})_i}
	\\
	\leq &\, \tfrac{2}{h} \normc[W^{1,\infty}(\hI)]{\hat{v}_i 
		- \realiz(\Phi^{\hat{\bsv}}_{\tau} )_i}
	\leq 
		\tfrac{2}{h}{\tau} \sum_{\ell\geq 2} |\hat{v}_{i,\ell}|
	\leq
		\tfrac{1}{2h} \eps (2/h)^{1/r'} \min_{v\in\bbP_n: \atop v'' = {v_i}''} \normc[L^{r'}(I)]{v}
	.
\end{align*}
Using exactness of $\Phi^{\bsv,I}_{\eps}$ in the endpoints of the interval,
H\"older's inequality implies that for all $i=1,\ldots,{N_{\bsv}}$, 
$1\leq r<\infty$
and $t=0,1$
\begin{align*}
(\tfrac2h)^{1-t}
		\snormc[W^{t,r}(I)]{v_i-\realiz(\Phi^{\bsv,I}_{\eps})_i}
	\leq &\, h^{1/r} (\tfrac2h)^{1-t}
	\snormc[W^{t,\infty}(I)]{v_i-\realiz(\Phi^{\bsv,I}_{\eps})_i}
	\\
	\leq &\, 2(h/2)^{1/r} \tfrac{1}{2h} \eps (2/h)^{1/r'} 
		\min_{v\in\bbP_n: \atop v'' = {v_i}''} \normc[L^{r'}(I)]{v} 
	\\
	= &\, \tfrac12 \eps (2/h)^{1+1/r'-1/r} \min_{v\in\bbP_n: \atop v'' = {v_i}''} \normc[L^{r'}(I)]{v} 
	,
\end{align*}
which is the first inequality in Equation \eqref{eq:relupolynomcorerr}.
For all $i=1,\ldots,{N_{\bsv}}$, we take $v := v_i - v_i(a) \in\bbP_n$ 
in the minimum in Equation \eqref{eq:relupolynomcorerr}.
It satisfies $v' = v_i'$, $v'' = v_i''$ and $v(a) = 0$.
Poincar\'e's inequality 
gives
$\normc[L^{r'}(I)]{v} \leq h \snormc[W^{1,r'}(I)]{v} = h \snormc[W^{1,r'}(I)]{v_i}$,
which gives the second inequality in Equation \eqref{eq:relupolynomcorerr}.

It remains to estimate the network depth and size.
We use Proposition \ref{prop:relupolynom} with 
$C_{\mathrm n} \leq (n+1){N_{\bsv}}$ and ${\tau} = \eps/(4C_0 n^6)$,
which gives
\begin{align*}
\depth(\Phi^{\bsv,I}_{\eps})
	\leq &\, \depth(\Phi^{\hat{\bsv}}_{\tau}) + \depth(\Phi^P)
	\\
	\leq &\, \left( C_{\depth}(1+\log_2(n))\log_2({1/{\tau}})
		+ \tfrac{2}{3} C_{\depth} (\log_2(n))^3 + C (1+\log_2(n))^2 \right)
		+2
	\\
	\leq &\, C_{\depth}(1+\log_2(n)) \log_2({1}/\eps) 
		+ \tfrac{2}{3} C_{\depth} (\log_2(n))^3 + C (1+\log_2(n))^2
		,
	\\
\size(\Phi^{\bsv,I}_{\eps})
	\leq &\, \size(\Phi^{\hat{\bsv}}_{\tau}) + \sizefirst(\Phi^{\hat{\bsv}}_{\tau}) 
		+ \sizelast(\Phi^P) + \size(\Phi^P)
	\\
	\leq &\, \big( 4C_{\size} n \log_2({1/{\tau}}) + 16C_{\size} n\log_2(n)
		+ 6C_{\depth} (1+\log_2(n))^2 \log_2({1/{\tau}}) 
	\\
	&\, + C (n+C_{\mathrm n}) \big)
		+ \max\{C_{\mathrm n}, 4\log_2(n)+4\}
		+ 3 + 7
	\\
	\leq &\, 4C_{\size} n \log_2({1}/\eps) + 40C_{\size} n\log_2(n)
		+ 6C_{\depth} (1+\log_2(n))^2 \log_2({1}/\eps) 
		+ C n {N_{\bsv}}
		,
	\\
\sizefirst(\Phi^{\bsv,I}_{\eps})
	\leq &\, \sizefirst(\Phi^P)
	\leq 4,
	\\
\sizelast(\Phi^{\bsv,I}_{\eps})
	\leq &\, \sizelast(\Phi^{\hat{\bsv}}_{\tau})
	\leq 2(n+1){N_{\bsv}}
	.
\end{align*}
The weights of $\Phi^P$ are independent of $(v_i)_{i=1}^{N_\bsv}$.
The remaining statements about the NN weights 
follow directly from Proposition \ref{prop:relupolynom}.
\end{proof}
\subsection{ReLU emulation of piecewise polynomial functions}
\label{sec:relupwpolynom}

Using Corollary \ref{cor:relupolynomH10},
the dependence of the network size on the polynomial degree in \cite[Proposition 5.1]{OPS2020} 
is improved in Proposition \ref{prop:relupwpolynom} below.

Based on Proposition \ref{prop:relupwpolynom}, 
improvements of some results in \cite[Section 5]{OPS2020} directly follow.
We present them in Section \ref{sec:FESpaces}
without proof,
because the proofs are completely analogous to those in \cite{OPS2020}.
Other results from \cite[Section 5]{OPS2020} do not improve,
because for those results it holds that $p \propto (1+\log(1/\eps))$,
which means that $p^2 + p \log(1/\eps)$ and $p\log(p) + p\log(1/\eps)$ are of the same order.

Below, we derive a result on the ReLU DNN approximation 
of continuous, piecewise polynomial functions on an arbitrary bounded interval $I$,
analogous to \cite[Proposition 5.1]{OPS2020}.
For continuous, piecewise polynomial functions
we use the notation from Section \ref{sec:CPwLoned}.
As for the ReLU emulation of polynomials in the previous section,
weights and biases in the hidden layers 
are independent of the approximated function.
For weights and biases in the output layer,
explicit numerical expressions based on the inverse fast Fourier transform
are provided in Remark \ref{rem:weightformulas}.
A quick extension of the error bounds to fractional order Sobolev spaces
is given in Corollary \ref{cor:fracsobrate}.

\begin{proposition}[{\cite[Proposition 6.3.1]{JOdiss}}]
\label{prop:relupwpolynom}
There exists a constant $C>0$ such that the following holds:
For $-\infty < a < b < \infty$, let $I := (a,b)$.
For all $\bm p = (p_i)_{i\in\{1,\ldots,N\}} \in \N^N$,
all partitions $\Tcal$ of $I$ into $N$ open, disjoint, connected subintervals 
$I_i = (x_{i-1},x_i)$ of length $h_i = x_i - x_{i-1}$, $i=1,\ldots,N$,
with $h  =  \max_{i\in\{1,\ldots,N\}} h_i$
and for all $v\in S_{\bm p} (I,\Tcal)$, for $0<\eps< 1$ there exist ReLU NNs
$\{\Phi^{v,\Tcal,\bm p}_{\eps}\}_{\eps\in(0,1)}$ 
such that for all $1\leq r, r' \leq \infty$ holds
\begin{align*}
(2/h_i)^{1-t} \snormc[W^{t,r}(I_i)]{v-\realiz\left(\Phi^{v,\Tcal,\bm p}_{\eps}\right)}
	\leq &\, 
		\eps \tfrac12 (2/h_i)^{1+1/r'-1/r} \min_{u\in\bbP_{p_i}: \atop u'' = v''|_{I_i}} \normc[L^{r'}(I_i)]{u},
	\\
	&\, \qquad \text{ for all } i=1,\ldots,N \text{ and } t = 0,1,
	\\
\tfrac1h \normc[L^{r}(I)]{v-\realiz\left(\Phi^{v,\Tcal,\bm p}_{\eps}\right)}
	\leq &\,
	\snormc[W^{1,r}(I)]{v-\realiz\left(\Phi^{v,\Tcal,\bm p}_{\eps}\right)}
	\leq 
		\tfrac12 \eps \snormc[W^{1,r}(I)]{v}	
	,
\end{align*}
\begin{align*}
\depth\left(\Phi^{v,\Tcal,\bm p}_{\eps}\right)
	\leq &\, \max_{i=1}^N \left( C_{\depth}(1+\log_2(p_i)) \log_2\big(1/\eps\big)
		+ \tfrac{2}{3} C_{\depth} (\log_2(p_i))^3 + C (1+\log_2(p_i))^2 \right)
	,
	\\
\size\left(\Phi^{v,\Tcal,\bm p}_{\eps}\right)
	\leq &\, \sum_{i=1}^N p_i \big( 4C_{\size} \log_2\big(1/\eps\big) 
		+ 40C_{\size} \log_2(p_i) \big)
	\\ 
	&\, + \sum_{i=1}^N 6C_{\depth} (1+\log_2(p_i))^2 \log_2\big(1/\eps\big)
		+ C \sum_{i=1}^N p_i 
	\\ 
	&\, + 2N \max_{i=1}^N \Big( C_{\depth}(1+\log_2(p_i)) \log_2\big(1/\eps\big)
			+ \tfrac{2}{3} C_{\depth} (\log_2(p_i))^3 
	\\ 
	&\, 	+ C (1+\log_2(p_i))^2 \Big)  
	,
	\\
\sizefirst\left(\Phi^{v,\Tcal,\bm p}_{\eps}\right)
	\leq &\, 2N+2
	,
	\\
\sizelast\left(\Phi^{v,\Tcal,\bm p}_{\eps}\right) 
	\leq &\, 3N+1 + 2\sum_{i=1}^N p_i
	.
\end{align*}

In addition, it holds that 
$\realiz\left( \Phi^{v,\Tcal,\bm p}_{\eps} \right)(x_j)=v(x_j)$ 
for all $j\in\{0,\ldots,N\}$.
The weights and biases in the hidden layers are independent of $v$.
The weights and biases in the output layer 
are linear combinations of the function values in the \CC points 
in $I_i$ for $i=1,\ldots,N$.
\end{proposition}

\begin{proof}
As in the proof of \cite[Proposition 5.1]{OPS2020},
we denote by $\bar{v}\in S_{\bm 1}(I,\Tcal)$ the continuous, piecewise linear interpolant of $v$,
which is exact in the nodes of the partition, i.e. $\bar{v}(x_i) = v(x_i)$ for all $i=0,\ldots,N$.
It can be emulated exactly by a NN of depth $2$ satisfying
$\size(\Phi^{\bar{v}}) \leq 3N+1$, 
$\sizefirst(\Phi^{\bar{v}}) \leq 2N$ and
$\sizelast(\Phi^{\bar{v}}) \leq N+1$,
see e.g. \cite[Lemma 3.1]{OPS2020}.

On each interval $I_i$, $i=1,\ldots,N$, 
we approximate the difference $v-\bar{v}\in S_{\bm p}(I,\Tcal)$ 
with $\Phi^{v-\bar{v},I_i}_{\eps_i}$ from Corollary \ref{cor:relupolynomH10}
for $\eps_i = \eps/4$, 
with ${N_{\bsv}}=1$ in the notation of the corollary.
It holds for all $1 \leq r, r' \leq \infty$ that
\begin{align}
\nonumber
\realiz(\Phi^{v-\bar{v},I_i}_{\eps_i}) |_{\R\setminus I_i}
	= &\, 0,
	\\\nonumber
(2/h_i)^{1-t} \snormc[W^{t,r}(I_i)]{(v-\bar{v})-\realiz(\Phi^{v-\bar{v},I_i}_{\eps_i})}
	\leq &\, \tfrac12 \eps_i (2/h_i)^{1+1/r'-1/r} 
		\min_{u\in\bbP_{p_i}: \atop u'' = (v-\bar{v})''|_{I_i}} \normc[L^{r'}(I_i)]{u}
	\\
	= &\, \tfrac12 \eps_i (2/h_i)^{1+1/r'-1/r} 
		\min_{u\in\bbP_{p_i}: \atop u'' = v''|_{I_i}} \normc[L^{r'}(I_i)]{u}
	,
	\label{eq:relupwpolynompferr}
\end{align}
\begin{align}
\nonumber
\depth(\Phi^{v-\bar{v},I_i}_{\eps_i})
	\leq &\, C_{\depth}(1+\log_2(p_i)) \log_2\big(1/\eps_i\big) 
		+ \tfrac{2}{3} C_{\depth} (\log_2(p_i))^3 + C (1+\log_2(p_i))^2,
	\\\nonumber
\size(\Phi^{v-\bar{v},I_i}_{\eps_i})
	\leq &\, 4C_{\size} p_i \log_2\big(1/\eps_i\big) + 40C_{\size} p_i\log_2(p_i)
		+ 6C_{\depth} (1+\log_2(p_i))^2 \log_2\big(1/\eps_i\big) 
	\\\nonumber
	&\, 
		+ C p_i,
	\\\nonumber
\sizefirst(\Phi^{v-\bar{v},I_i}_{\eps_i})
	\leq &\, 4,
	\\\nonumber
\sizelast(\Phi^{v-\bar{v},I_i}_{\eps_i})
	\leq &\, 2(p_i+1).
\end{align}

Let $(\ell_i)_{i=0}^{N} \in \N^{N+1}$ be such that 
$\depth(\Phi^{\bar{v}}) +\ell_0 
	= \max\{ \depth(\Phi^{\bar{v}}), \max_{i=1}^N \depth(\Phi^{v-\bar{v},I_i}_{\eps_i}) \} + 1
	= \depth(\Phi^{v-\bar{v},I_i}_{\eps_i}) +\ell_i$ for all $i=1,\ldots,N$,
so that $\min_{i=0}^N \ell_i = 1$ and $\sum_{i=0}^N \ell_i \leq 1+ N \max_{i=0}^N \ell_i$.
We then define $\Phi^{v,\Tcal,\bm p}_{\eps}$ as follows,
with $\bsone := (1,\ldots,1)^\top\in\R^{N+1}$ 
and the concatenation from Definition \ref{def:pvconc}
\begin{align*}
\Phi^{v,\Tcal,\bm p}_{\eps}
	:=  
	((\bsone^\top,0)) \bullet \Parallel{
		\Phi^{\bar{v}} \sconc \Phi^{\Id}_{1,\ell_0}, 
		\Phi^{v-\bar{v},I_1}_{\eps_1} \sconc \Phi^{\Id}_{1,\ell_1}, \ldots, 
		\Phi^{v-\bar{v},I_N}_{\eps_N} \sconc \Phi^{\Id}_{1,\ell_N} }.
\end{align*}

The realization satisfies
\begin{align*}
\realiz(\Phi^{v,\Tcal,\bm p}_{\eps})
	= &\, \bar{v} + \sum_{i=1}^N \realiz(\Phi^{v-\bar{v},I_i}_{\eps_i}).
\end{align*}
The first error bound in the proposition follows from Equation \eqref{eq:relupwpolynompferr}.
For the second error estimate, we use that 
$\snorm[W^{1,r}(I_i)]{\bar{v}} \leq \snorm[W^{1,r}(I_i)]{v}$,
which we next show, using Jensen's inequality in the second to last step:
\begin{align*}
\snorm[W^{1,r}(I_i)]{\bar{v}}^{r}
	= &\, \int_{I_i} \snorm{\bar{v}'}^{r} \dd x
	= \int_{I_i} \snormc{ \tfrac{1}{h_i} \int_{I_i} v'(y) \dd y }^{r} \dd x
	= h_i \snormc{ \tfrac{1}{h_i} \int_{I_i} v' \dd y }^{r}
	\\
	\leq &\, h_i \left( \tfrac{1}{h_i} \int_{I_i} \snorm{v'} \dd y \right)^{r}
	\leq h_i \tfrac{1}{h_i} \int_{I_i} \snorm{v'}^{r} \dd y
	= \snorm[W^{1,r}(I_i)]{v}^{r}
	.
\end{align*}
Using this fact, we obtain
\begin{align*}
\snormc[W^{1,r}(I)]{v-\realiz\left(\Phi^{v,\Tcal,\bm p}_{\eps}\right)}
	= &\,\snormc[W^{1,r}(I)]{ (v-\bar{v}) - \sum_{i=1}^N \realiz(\Phi^{v-\bar{v},I_i}_{\eps_i}) }
	\\
	= &\, 
		\left(\sum_{i=1}^N \snormc[W^{1,r}(I_i)]{(v-\bar{v})-\realiz(\Phi^{v-\bar{v},I_i}_{\eps_i})}^{r}
			\right)^{1/r}
	\\
	\leq &\, 
		\left(  \sum_{i=1}^N \eps_i^{r} \snorm[W^{1,r}(I_i)]{v-\bar{v}}^{r}  \right)^{1/r}
	\\
	\leq &\, 
		\left(  \sum_{i=1}^N \eps^{r} 4^{-r} 
			( \snorm[W^{1,r}(I_i)]{v} + \snorm[W^{1,r}(I_i)]{\bar{v}} )^{r}  \right)^{1/r}
	\leq \tfrac12 \eps \snorm[W^{1,r}(I)]{v}
\end{align*}
and by Poincar\'e's inequality
\begin{align*}
(\tfrac1h)^{r} \normc[L^{r}(I)]{v-\realiz\left(\Phi^{v,\Tcal,\bm p}_{\eps}\right)}^{r}
	\leq &\, \sum_{i=1}^N (\tfrac{1}{h_i})^{r}
		\normc[L^{r}(I_i)]{v-\realiz\left(\Phi^{v,\Tcal,\bm p}_{\eps}\right)}^{r}
	\\
	\leq &\, \sum_{i=1}^N \snormc[W^{1,r}(I_i)]{v-\realiz\left(\Phi^{v,\Tcal,\bm p}_{\eps}\right)}^{r}
	= \snormc[W^{1,r}(I)]{v-\realiz\left(\Phi^{v,\Tcal,\bm p}_{\eps}\right)}^{r}
.
\end{align*}
The network depth and size can be estimated in terms of $p_{\max} := \max_{i=1}^N p_i$:
\begin{align*}
\depth\left(\Phi^{v,\Tcal,\bm p}_{\eps}\right)
	\leq &\, 
		\max_{i=0}^N \ell_i + 1
	\\
	\leq &\,
		\Big( C_{\depth}(1+\log_2(p_{\max})) \log_2\big(1/\eps\big)
		+ \tfrac{2}{3} C_{\depth} (\log_2(p_{\max}))^3 
	\\
	&\, + C (1+\log_2(p_{\max}))^2 \Big)
	,
	\\
\size\left(\Phi^{v,\Tcal,\bm p}_{\eps}\right)
	\leq &\, 
		\size(\Phi^{\bar{v}} \sconc \Phi^{\Id}_{1,\ell_0})
		+ \sum_{i=1}^N \size(\Phi^{v-\bar{v},I_i}_{\eps_i} \sconc \Phi^{\Id}_{1,\ell_i}) 
	\\
	\leq &\, 
		\size(\Phi^{\bar{v}})
		+ \sizefirst(\Phi^{\bar{v}})
		+ \sizelast(\Phi^{\Id}_{1,\ell_0})
		+ \size(\Phi^{\Id}_{1,\ell_0})
	\\
	&\,
		+ \sum_{i=1}^N \Big( 
			\size( \Phi^{v-\bar{v},I_i}_{\eps_i}) 
		+ \sizefirst( \Phi^{v-\bar{v},I_i}_{\eps_i})
		+ \sizelast(\Phi^{\Id}_{1,\ell_i})
		+ \size(\Phi^{\Id}_{1,\ell_i})\Big)
	\\
	\leq &\, (3N+1) + 2N + 2 + 2\ell_0
	\\
	&\, + \sum_{i=1}^N \Big( 4C_{\size} p_i \log_2\big(1/\eps_i\big) + 40C_{\size} p_i\log_2(p_i)
	\\
	&\, 
		+ 6C_{\depth} (1+\log_2(p_i))^2 \log_2\big(1/\eps_i\big)
		+ C p_i \Big)
		+ 4N + 2N + 2\sum_{i=1}^N \ell_i
	\\
	\leq &\, \sum_{i=1}^N p_i \big( 4C_{\size} \log_2\big(1/\eps\big) 
		+ 40C_{\size} \log_2(p_i) \big)
	\\
	&\, + \sum_{i=1}^N 6C_{\depth} (1+\log_2(p_i))^2 \log_2\big(1/\eps\big)
		+ C \sum_{i=1}^N p_i
	\\ 
	&\,
		+ 2N \Big( C_{\depth}(1+\log_2(p_{\max})) \log_2\big(1/\eps\big)
			+ \tfrac{2}{3} C_{\depth} (\log_2(p_{\max}))^3 
	\\ 
	&\, 	
			+ C (1+\log_2(p_{\max}))^2 \Big)  
	,
	\\
\sizefirst\left(\Phi^{v,\Tcal,\bm p}_{\eps}\right)
	\leq &\,
		\sum_{i=0}^N \sizefirst(\Phi^{\Id}_{1,\ell_i}) 
		\leq 2N+2
	,
	\\
\sizelast\left(\Phi^{v,\Tcal,\bm p}_{\eps}\right)
	\leq &\, 
		\sizelast(\Phi^{\bar{v}}) + \sum_{i=1}^N \sizelast(\Phi^{v-\bar{v},I_i}_{\eps_i})
	\leq (N+1) + \sum_{i=1}^N 2(p_i+1) 
	\\
	= &\, 3N+1 + 2\sum_{i=1}^N p_i
	.
\end{align*}
The realization is exact in the points $(x_j)_{j=0}^N$
because $\Phi^{\bar{v}}$ emulates $\bar{v}$ exactly
and because $\realiz(\Phi^{v-\bar{v},I_i}_{\eps_i})$ vanishes in $(x_j)_{j=0}^N$
for all $i=1,\ldots,N$.
Given $\Tcal$,
the hidden layer weights and biases are independent of $v$,
because those of $\Phi^{\bar{v}}$ are independent of $\bar{v}$ 
(see the proof of \cite[Lemma 3.1]{OPS2020})
and those of $\realiz(\Phi^{v-\bar{v},I_i}_{\eps_i})$ are independent of $v-\bar{v}$,
as stated in Corollary \ref{cor:relupolynomH10}.
The weights and biases in the output layer
can be computed as linear combinations of the function values of $v$
in the \CC grids on the subintervals of $\Tcal$.
Explicit formulas are given in Remark \ref{rem:weightformulas} below.
\end{proof}

\begin{remark}
\label{rem:weightformulas}
The weights of $\Phi^{v,\Tcal,\bm p}_{\eps}$ in the output layer
can be computed easily from the function values of $v$
in the \CC grids on the subintervals of the partition $\Tcal$
using the inverse fast Fourier transform.
In this remark we give explicit formulas.

The weights in the final layer of $\Phi^{\bar{v}}$ are 
$\tfrac{v(x_1) - v(x_0)}{x_1-x_0}$ and 
$\tfrac{v(x_i) - v(x_{i-1})}{x_i-x_{i-1}} - \tfrac{v(x_{i-1}) - v(x_{i-2})}{x_{i-1}-x_{i-2}}$
for $i=2,\ldots,N$.
The bias equals $v(x_0)$.
For $i=1,\ldots,N$, the weights and biases in the final layer of $\Phi^{v-\bar{v},I_i}_{\eps_i}$ 
are, for $w_i := (v-\bar{v}) |_{I_i}$,
\begin{subequations}
\label{eq:ifftproprelupwpolynom}
\begin{align}
\nonumber
\widehat{(w_i)}_{\ell} := &\, 
\frac{2^{\mathds{1}_{\{1,\ldots,{p_i}-1\}}(\ell)}}{2{p_i}} \sum_{j=0}^{2{p_i}-1} (v-\bar{v}) (x^{\cc,p_i}_{I_i,j}) \cos(\ell j \pi / {p_i}) 
	\\
	= &\, 2^{\mathds{1}_{\{1,\ldots,{p_i}-1\}}(\ell)} \ifft\left( \{ (v-\bar{v}) (x^{\cc,p_i}_{I_i,j}) \}_{j=0}^{2{p_i}-1} \right)_\ell,
	\qquad 
	\ell=0,\ldots,p_i
	,
\label{eq:ifftproprelupwpolynomi}
\end{align}
where
$(v-\bar{v}) (x^{\cc,p_i}_{I_i,j})$ 
can be expressed in terms of function values of $v$ as
\begin{align}
\label{eq:ifftproprelupwpolynomii}
(v-\bar{v}) (x^{\cc,p_i}_{I_i,j}) 
	= v(x^{\cc,p_i}_{I_i,j}) 
	- \frac{ x_i-x^{\cc,p_i}_{I_i,j} }{x_i-x_{i-1}} v(x_{i-1})
	- \frac{ x^{\cc,p_i}_{I_i,j}-x_{i-1} }{x_i-x_{i-1}} v(x_i).
\end{align}
\end{subequations}

Because $\bar{v}|_{I_i} \in \bbP_1(I_i)$,
the \Cheb coefficients $ \widehat{(w_i)}_{\ell}$ of $(v-\bar{v}) |_{I_i}$ of degree $\ell>1$
equal those of $v$. 
Therefore, for $\ell>1$, $v - \bar{v}$ can be replaced by $v$ 
in \eqref{eq:ifftproprelupwpolynomi}.

For $\ell=0,1$, $\widehat{(w_i)}_{\ell}$ 
can be determined from $\big(\widehat{(w_i)}_{\ell}\big)_{\ell>1}$
by writing the conditions
$(v-\bar{v})(x_{i-1}) = 0 = (v-\bar{v})(x_{i})$ 
in terms of the \Cheb coefficients,
using that $T_\ell(\pm1) = (\pm1)^\ell$ for all $\ell\in\N_0$.
This gives
\begin{align}
\label{eq:ifftproprelupwpolynomlowdeg}
\sum_{\ell=0}^{p_i} (-1)^\ell \widehat{(w_i)}_{\ell} 
	= 0 
	= \sum_{\ell=0}^{p_i} \widehat{(w_i)}_{\ell}
.
\end{align}
\end{remark}

With the theory of interpolation spaces,
we directly obtain the following corollary of Proposition \ref{prop:relupwpolynom}
when the error is measured in fractional Sobolev norms.
\begin{corollary}
\label{cor:fracsobrate}
Under the assumptions of Proposition \ref{prop:relupwpolynom},
for all 
$v\in S_{\bm p} (I,\Tcal)$,
$0<\eps< 1$,
$s\in[0,1]$ 
and 
$r\in[1,\infty]$ 
holds
\begin{align*}
\normc[W^{s,r}(I)]{v-\realiz\left(\Phi^{v,\Tcal,\bm p}_{\eps}\right)}
	\leq &\, 
		C(r,s,I) h^{1-s} \eps \snormc[W^{1,r}(I)]{v}.
\end{align*}

\end{corollary}
\begin{proof}
First, we conclude the statement for $s\in\{0,1\}$ 
from the second error bound in Proposition \ref{prop:relupwpolynom}:
\begin{align*}
\normc[L^{r}(I)]{v-\realiz\left(\Phi^{v,\Tcal,\bm p}_{\eps}\right)}
	\leq &\, h \eps \snormc[W^{1,r}(I)]{v}
	,
	\qquad
\normc[W^{1,r}(I)]{v-\realiz\left(\Phi^{v,\Tcal,\bm p}_{\eps}\right)} 
	\leq	\eps \snormc[W^{1,r}(I)]{v}
	.
\end{align*}

For $s\in(0,1)$,
we use \cite[Example 1.7]{Lunardi2018}
and \cite[Theorem 14.2.3]{BS2008}\footnote{
\cite[Theorem 14.2.3]{BS2008} is stated for $1\leq r < \infty$ 
(which is denoted by $p$ in the notation of the reference).
The statement also holds for $r=\infty$.
For function spaces over the domain $\R$,
a proof is given in \cite[Example 1.8]{Lunardi2018}
and the discussion after the proof of that example.
The statement on the bounded interval $I$ 
follows from the statement on $\R$
by the same steps as in the proof of \cite[Theorem 14.2.3]{BS2008}.
},
which gives
\begin{align*}
\normc[W^{s,r}(I)]{v-\realiz\left(\Phi^{v,\Tcal,\bm p}_{\eps}\right)} 
	\leq &\, C(r,s,I)
\normc[L^r(I)]{v-\realiz\left(\Phi^{v,\Tcal,\bm p}_{\eps}\right)}^{1-s}
\normc[W^{1,r}(I)]{w-\realiz\left(\Phi^{w,\Tcal,\bm p}_{\eps}\right)}^s
	\\
	\leq &\,
		C(r,s,I) h^{1-s} \eps^{1-s} \snormc[W^{1,r}(I)]{v}^{1-s} \eps^s \snormc[W^{1,r}(I)]{v}^s
	= C(r,s,I) h^{1-s} \eps \snormc[W^{1,r}(I)]{v}
	.
\end{align*}
\end{proof}
\section{ReLU emulation of univariate finite element spaces} 
\label{sec:FESpaces}
Based on Proposition \ref{prop:relupwpolynom},
improved dependence of the NN depth and size on the polynomial degree 
directly follows for several emulation rate bounds from \cite{OPS2020}
for univariate $h$-, $p$- and $hp$-Finite Element spaces.
We state them in 
Sections~\ref{sec:FreeKnotSplApr}, \ref{sec:spectralmeth} and \ref{sec:hpFEM}.
The former two are stated, without detailing proofs.
Their proofs are completely analogous to those of the corresponding results in \cite{OPS2020},
with improved dependence of NN size on the polynomial degree due to 
the emulation of \Cheb polynomials instead of monomials.
\subsection{Free-knot splines}
\label{sec:FreeKnotSplApr}
The following analogue of \cite[Theorem 5.4]{OPS2020} obtains,
for fixed parameters $q,q',t,s,s',p$,
and up to a factor $\log(N)$ in the network size, 
the classical adaptive ``$h$-FEM'' rate $N^{-(s-s')}$,
with a network depth of the order $\log(N)$.
Similar results, which also apply to multivariate functions, 
but are restricted to error bounds in $L^q$, $q\in(0,\infty]$,
were obtained in \cite[Proposition 1 and Theorem 1]{Suzuki2019}.

We first state a spline approximation result.
Combined with Proposition \ref{prop:relupwpolynom},
we obtain ReLU NN approximations in Corollary \ref{cor:hFEMReLUResult}.
The proof of that corollary 
is a straightforward application of the two propositions,
analogous to the proof of \cite[Theorem 5.4]{OPS2020}.

\begin{proposition}[{\cite[Theorems 3 and 6]{oswald1990degree}}]
\label{prop:BesoVApprox}
Let $I = (0,1)$, $q,q',t,t',s,s' \in (0, \infty]$, $p \in \N$, 
and 
$$
q < q',\quad  s < p + 1/q, \quad s' < s- 1/q + 1/q'.
$$
Then, there exists a $C_3 \coloneqq C(q,q',t,t',s,s',p) > 0$ and, 
for every $N \in \N$ and every $f$ in $B_{q,t}^s(I)$, 
for $\bsp = (p,\ldots,p) \in \N^N$,
there exist
a partition $\Tcal$ of $I$ with $N$ elements and
$\varphi \in S_{\bsp}(I,\Tcal)$ 
such that 
\begin{align}
\label{eq:BesoVRate}
\left\|f- \varphi\right\|_{B^{s'}_{q',t'}(I)} \leq C_3 N^{-(s-s')}\|f\|_{B^{s}_{q,t}(I)}.
\end{align}
Moreover,
\begin{align} \label{eq:BoundOnSplineNorm}
\left\|\varphi\right\|_{B^{s}_{q,t}(I)} \leq C_3 \|f\|_{B^{s}_{q,t}(I)}.
\end{align}
\end{proposition}

\begin{corollary}
\label{cor:hFEMReLUResult}
Let 
$I = (0,1)$,
$0 < q < q' \leq \infty$, $q'\geq 1$, $0 < t \leq \infty$. 
Let $p \in \N$, $0 < s' \leq 1 < s < p + 1/q$, $1 - 1/q' < s - 1/q$ 
and $s'<1$ if $p=1$ and  $q'=\infty$. 

Then, 
there exists a constant $C_*(q,q',t,s,s',p) > 0$ and, 
for every $N \in \N$ and every $f \in B_{q,t}^s(I)$, 
there exists a NN $\Phi^N_f$ such that 
\begin{align*}
\left\|f- \realiz\left( \Phi^N_f \right)\right\|_{W^{s',q'}(I)} 
     &\leq C_* N^{-(s-s')}\|f\|_{B^{s}_{q,t}(I)} 
\end{align*}
and
\begin{align*}
\depth\left(\Phi^N_f\right)
	\leq &\, C_{\depth} (1+\log_2(p)) (s-s') \log_2\left(N\right) 
		+ \tfrac23 C_{\depth} \log_2^3(p)
		+ C\left(1+\log_2(p) \right)^2,
	\\
\size\left(\Phi^N_f\right) 
	\leq &\, 4C_{\size} (s-s') N \log_2\left(N\right) p 
		+ 40C_{\size} N p\log_2(p)
	\\
	&\, + 6C_{\depth} (s-s') \log_2\left(N\right) N ( 1+ \log_2(p) )^2
		+ C N p 
	\\
	&\, + 2N\Big( C_{\depth} (1+\log_2(p)) (s-s') \log_2\left(N\right) 
		+ \tfrac23 C_{\depth} \log_2^3(p)
		+ C\left(1+\log_2(p)\right)^2\Big)
	,
	\\
\sizefirst\left(\Phi^N_f\right)
	\leq &\, 2N+2
	,
	\\
\sizelast\left(\Phi^N_f\right)
	\leq &\, N(2p+3)+1
.
	\end{align*}
\end{corollary}

The improvement with respect to \cite[Theorem 5.4]{OPS2020}
is a better dependence of the network depth and size 
on the polynomial degree $p$, 
namely $C(1+\log(p))^3$ and $Cp(1+\log(p))$
instead of $Cp(1+\log(p))$ and $Cp^2$ in \cite[Theorem 5.4]{OPS2020}.
%
\subsection{Spectral methods}
\label{sec:spectralmeth}

Corollary \ref{cor:relupfemTcal} provides ReLU NN emulation rate bounds 
for spectral and so-called ``$p$-version Finite Element'' methods.
The ReLU NN constructions based on \Cheb polynomials 
imply \emph{improved dependence of the NN size on the polynomial degree 
and superior numerical stability},
in comparison to the construction in \cite[Theorem 5.8]{OPS2020}.
In particular,
the ReLU NN size is now proportional to the number of degrees of freedom $Np+1$,
up to a polylogarithmic factor.
In 
\cite[Proposition 1 and Theorem 1]{Suzuki2019} 
and 
\cite[Corollary 4.2]{GKP2020}, 
similar statements with the same approximation rates were obtained,
but based on a different argument. 
There, 
the approximation was based on the NN emulation of B-splines 
and of averaged Taylor polynomials, respectively, 
of fixed polynomial degree, and on partitions that are refined to reduce the error.
The result below is based on the approximation of piecewise polynomials 
for which both the polynomial degree and the mesh can be refined to reduce the error.
We observe that 
\cite[Proposition 1 and Theorem 1]{Suzuki2019} and \cite[Corollary 4.2]{GKP2020} 
apply more generally, in particular also to multivariate functions.

First, we state the finite element result 
on which Corollary \ref{cor:relupfemTcal} is based.

\begin{proposition}[{{\cite[Theorem 3.17]{Schwab1998}, \cite[Remark 5.7]{OPS2020}}}]
\label{prop:pfemTcal}
Let $\bar{s}\in\N_0$.
Then, there exists a constant $c(\bar{s}) > 0$
such that
for all partitions $\Tcal$ of $I=(0,1)$ with $N$ elements, 
for all $u\in H^{\bar{s}+1}(I)$ and all $p\in\N$,
with 
$\bm p \coloneqq (p,\ldots,p) \in \N^N$,
there exists $v\in S_{\bm p}(I,\Tcal)$
such that for all $s\in\N_0$ satisfying $s\leq \min\{\bar{s},p\}$
\begin{align*}
\normc[H^1(I)]{u - v} \leq &\, c(\bar{s}) h^s p^{-s} \snormc[H^{s+1}(I)]{u}
,
\end{align*}
where $h = \max_{i\in\{1,\ldots,N\}} h_i$ 
is the maximum element size, defined in Section \ref{sec:CPwLoned}.

In addition, $\snormc[H^1(I)]{v} \leq \snormc[H^1(I)]{u}$.
\end{proposition}

\begin{corollary}
\label{cor:relupfemTcal}
Let $I=(0,1)$, $\bar{s}\in\N_0$ and $u\in H^{\bar{s}+1}(I)$. 

Then, there is an absolute constant $C>0$ and 
a constant $c(\bar{s}) > 0$ (depending only on $\bar{s}$)
such that,
for all $p,N\in\N$ 
there exists a NN $\Phi^{u,N,p}$ 
such that for all $s\in\N_0$ satisfying 
$s\leq \min\{\bar{s},p\}$
\begin{align*}
\normc[H^1(I)]{u-\realiz(\Phi^{u,N,p})} 
	\leq &\, c(\bar{s}) (Np)^{-s} \normc[H^{s+1}(I)]{u}
	,
	\\
\depth(\Phi^{u,N,p})
	\leq &\, C_{\depth} \bar{s} (1+\log_2(p)) \log_2(Np) 
		+ \tfrac23 C_{\depth} \log_2^3(p)
		+ C\left(1+\log_2(p) \right)^2,
	\\
\size(\Phi^{u,N,p})
	\leq &\, N[ 4C_{\size} \bar{s} p \log_2\left(Np\right) 
		+ 40C_{\size} p\log_2(p)
	\\
	&\, + 6C_{\depth} \bar{s} \log_2(Np) \left(1+ \log_2(p) \right)^2
		+ Cp 
	\\
	&\, + 2C_{\depth} \bar{s} \log_2(Np) \left(1+ \log_2(p) \right) ],
	\\
\sizefirst(\Phi^{u,N,p})
	\leq &\, 2N+2
	,
	\\
\sizelast(\Phi^{u,N,p})
	\leq &\, N(2p+3)+1
.
\end{align*}
\end{corollary}
This result 
follows from Proposition \ref{prop:pfemTcal}
if we choose $\Tcal$ 
to be the uniform partition of $I$ into $N$ intervals
(which is the optimal choice of $\Tcal$).
The proof is analogous to that of \cite[Theorem 5.8]{OPS2020}.
We improve upon that result by achieving the same error with a smaller network.
Up to logarithmic factors,
the NN size is $O(Np)$, 
for $N\to\infty$ and/or $p\to\infty$, rather than $O(Np^2)$ in \cite[Theorem 5.8]{OPS2020}.
\subsection{Approximation of weighted Gevrey regular functions}
\label{sec:hpFEM}
We consider approximation rates of ReLU NNs via so-called 
``$hp$-approximations''\footnote{Also known as ``variable-degree, free-knot splines''.}
of functions on $I=(0,1)$ which may have a singularity at $x=0$ 
and which belong to a weighted Gevrey class $\Gcal^{2,\delta}_\beta(I)$
defined in \eqref{eq:defGevrey}.
The following theorem is a generalization of \cite[Theorem 3.36]{Schwab1998},
stated for the analytic case $\delta=1$, 
to Gevrey regular functions for $\delta\geq1$.
It is based on the geometric partition $\Tcal_{\sigma,N}$ of $I = (0,1)$ 
introduced in Section \ref{sec:CPwLoned}. 
\begin{proposition}[{{\cite[Theorem 5.10]{OPS2020}}}] 
\label{prop:hp}
Let $\sigma,\beta\in(0,1)$, $\lambda \coloneqq \sigma^{-1}-1$, $\delta\geq1$, $u\in\Gcal^{2,\delta}_\beta(I)$ 
and $N\in\N$ be given.
For $\mu_0 \coloneqq \mu_0(\sigma,\beta,\delta,\du[u]) \coloneqq \max\left\{1,
\frac{\du[u]\lambda e^{1-\delta}}{2\sigma^{1-\beta}} \right\}$
and 
for $\mu>\mu_0$ let $\bm p=(p_i)_{i=1}^{N}\subset\N$ be defined as 
$p_1 \coloneqq 1$ and $p_i \coloneqq \lfloor \mu i^\delta \rfloor$ for $i\in\{2,\ldots,N\}$. 

Then, as $N\to \infty$, 
${\rm dim}(S_{\bm p}(I,\Tcal_{\sigma,N})) = O(N^{1+\delta})$.
Furthermore, 
there exists a continuous, piecewise polynomial function 
$v\in S_{\bm p}(I,\Tcal_{\sigma,N})$ 
such that $v(x_{i})=u(x_{i})$ for $i\in\{1,\ldots,N\}$ 
and such that 
for constants
$b({\sigma,\beta,\delta,\mu})>0$, 
$c(\beta,\sigma) := (1-\beta)\log(1/\sigma) > 0$
and $C(\sigma,\beta,\delta,\mu,\Cu[u],\du[u]) > 0$ 
(where $\Cu[u]$ and $\du[u]$ are as in Equation \eqref{eq:defGevrey})
it holds that
\begin{align*}
\normc[H^1(I)]{u-v}
	\leq &\, C \exp(-c(\beta,\sigma)N) 
        \leq C\exp\left(-b\; {\rm dim}(S_{\bm p}(I,\Tcal_{\sigma,N}))^{1/(1+\delta)}\right)
\;.
\end{align*}
\end{proposition}
This $hp$-approximation result implies, 
via the ReLU emulation of \Cheb polynomials, 
the following DNN expression rate bound.
For $\delta>1$, the exponent $1/(2+\delta)$ in the error bound of the theorem
is larger than the exponent $1/(2\delta+1)$ in \cite{OPS2020}.
We obtain a better exponential convergence rate bound.
\begin{theorem}
\label{thm:reluhp}
For all $\sigma,\beta \in (0,1)$, all $\delta\geq 1$,  
$\mu > \mu_0(\sigma,\beta,\delta,\du[u]) 
	$ $:= \max\left\{1,\frac{\du[u](\sigma^{-1}-1) e^{1-\delta}}{2\sigma^{1-\beta}} \right\}$,
and all $u\in\Gcal^{2,\delta}_{\beta}(I)$ 
there exist constants $C_* > 0$, $c>0$ 
and $b_*>0$
and
ReLU NNs $\{\Phi^{u,\sigma,N}\}_{N\in\N}$ 
of size $\size(\Phi^{u,\sigma,N})$ as in Definition~\ref{def:NeuralNetworks}
such that
\begin{align*}
\forall N\in\N:\quad 
\normc[H^1(I)]{u-\realiz(\Phi^{u,\sigma,N})}
	\leq &\,
			C_* \exp(-c(\beta,\sigma)N)  
    \leq 
    		C_* \exp\big(-b_* \size(\Phi^{u,\sigma,N})^{1/(2+\delta)} \big).
\end{align*}
Here, the constants $C_*,c,b_*$ satisfy
$c(\beta,\sigma) := (1-\beta)\log(1/\sigma)$,
$b_* := b_*(\sigma,\beta,\delta,\mu)>0$
and 
$C_* := C_*(\sigma,\beta,\delta,\mu,$ $\Cu[u],\du[u], \snormc[H^1(I)]{u})>0$,
and the ReLU NNs 
$\{\Phi^{u,\sigma,N}\}_{N\in\N}$ 
are such that 
\begin{align*}
\depth(\Phi^{u,\sigma,N})
	\leq &\, C_{\depth} c N ( \delta \log_2(N) + \log_2(\mu) +1)
		+ C(\delta,\mu) (1+\log_2(N) )^3
		,	
	\\
\size(\Phi^{u,\sigma,N})
	\leq &\, 4C_{\size} \mu c N^{\delta+2} 
		+ C(\sigma,\beta,\delta,\mu) \left( N^{\delta+1}\log_2(N) + N^2\log_2^{2}(N) \right)
	,
	\\
\sizefirst(\Phi^{u,\sigma,N}) 
        \leq &\, 2N+2
	,
	\\
\sizelast(\Phi^{u,\sigma,N})
	\leq &\, 2\mu N^{\delta+1} + 3N+1
.
\end{align*}
\end{theorem}

\begin{proof}
The proof is analogous to that of \cite[Theorem 5.12]{OPS2020}:
an $hp$ Finite Element approximation will be emulated by DNNs,
and the triangle inequality will imply ReLU NN expression rate bounds.

The $hp$-approximation $v\in S_{\bm p}(I,\Tcal_{\sigma,N})$ 
is as in Proposition \ref{prop:hp}, 
with $\bm p\subset\N$ defined by $p_1=1$ and with
polynomial degree vector
$p_i=\lfloor \mu i^\delta\rfloor$ for $i\in\{2,\ldots,N\}$.
With $\eps \coloneqq \exp(-cN)$
we use Proposition \ref{prop:relupwpolynom}
to construct the ReLU NN 
$\Phi^{u,\sigma,N} \coloneqq \Phi^{v,\Tcal_{\sigma,N},\bm p}_{\eps}$.

From $p_1=1$ it follows that
$\normc[H^1(I_{1,\sigma})]{v-\mathrm{R}(\Phi^{u,\sigma,N})}=0$,
because ReLU NN emulations of continuous, piecewise linear functions are exact
(this can also be seen by setting $u\equiv 0$ 
in the first error bound in Proposition \ref{prop:relupwpolynom}).

By \cite[Remark 5.11]{OPS2020}, 
which states that
$\snormc[H^1(I\backslash I_{1,\sigma})]{v} \leq \snormc[H^1(I\backslash I_{1,\sigma})]{u}$,
it follows that 
\begin{align*}
\normc[H^1(I)]{u-\mathrm{R}(\Phi^{u,\sigma,N})}
\leq &\, \normc[H^1(I)]{u - v}
		+ \normc[H^1(I)]{v-\mathrm{R}(\Phi^{v,\Tcal_{\sigma,M},\bm p}_{\eps})}
	\\
	\leq &\, C(\sigma,\beta,\delta,\mu,\Cu[u],\du[u]) \exp(-cN) + \exp(-cN) \snormc[H^1(I\backslash I_{\sigma,1})]{v}
	\\
	\leq &\, \big(C(\sigma,\beta,\delta,\mu,\Cu[u],\du[u]) + \snormc[H^1(I)]{u} \big) \exp(-cN)
	\\
	=: &\, C_* \exp(-cN) 
	,	
	\\
\depth(\Phi^{u,\sigma,N})
	\leq &\, \max_{i=1}^N \left( C_{\depth}(1+\log_2(\mu i^\delta)) \log_2\big(1/\eps\big)
		+ \tfrac{2}{3} C_{\depth} (\log_2(\mu i^\delta))^3 + C (1+\log_2(\mu i^\delta))^2 \right)
	\\
	\leq &\, C_{\depth} c N ( \delta \log_2(N) + \log_2(\mu) +1 )
		+ C(\delta,\mu) (1+\log_2(N) )^3
	,
	\\
\size(\Phi^{u,\sigma,N})
	\leq &\, \sum_{i=1}^N \mu i^\delta \big( 4C_{\size} \log_2\big(1/\eps\big) 
		+ 40C_{\size} \log_2(\mu i^\delta) \big)
	\\ 
	&\, + \sum_{i=1}^N 6C_{\depth} (1+\log_2(\mu i^\delta))^2 \log_2\big(1/\eps\big)
		+ C \sum_{i=1}^N \mu i^\delta
	\\ 
	&\, + 2N \max_{i=1}^N \Big( C_{\depth}(1+\log_2(\mu i^\delta)) \log_2\big(1/\eps\big)
			+ \tfrac{2}{3} C_{\depth} (\log_2(\mu i^\delta))^3 
	\\ 
	&\, 	+ C (1+\log_2(\mu i^\delta))^2 \Big)  
	\\
	\leq &\, 4C_{\size} \mu c N^{\delta+2} 
		+ C(\sigma,\beta,\delta,\mu) \left( N^{\delta+1}\log_2(N) + N^2\log_2^{2}(N) \right)
	,
	\\
\sizefirst(\Phi^{u,\sigma,N})
	\leq &\, 2N+2
	,
	\\
\sizelast(\Phi^{u,\sigma,N})
	\leq &\, 3N+1 + 2\sum_{i=1}^N \mu i^\delta
	\leq 2\mu N^{\delta+1} + 3N+1
	.
\end{align*}
To finish the error bound,
we rewrite the bound on the network size as
\[
c N \geq c C(\sigma,\beta,\delta,\mu)^{-1/(2+\delta)} \size(\Phi^{u,\sigma,N})^{1/(2+\delta)} =: b_* \size(\Phi^{u,\sigma,N})^{1/(2+\delta)}
\]
and obtain that
$\exp(-cN) \leq \exp\big(-b_* \size(\Phi^{u,\sigma,N})^{1/(2+\delta)} \big)$.
\end{proof}

The results in this section are based on $H^1$-projections
in \cite[Section 3.3]{Schwab1998}.
For weighted analytic functions, i.e. $\delta=1$, 
analogous results based on 
elementwise interpolation in the \CC nodes
are proved in \cite[Section 6.4]{JOdiss},
with error bounds in $W^{1,r}(I)$ for $r\in(1,\infty]$.

Our results generalize in a straightforward way 
to functions with a finite number of algebraic singularities,
similarly to \cite[Remark 2.5.6]{JOdiss}.
Generalizations to higher dimensions also hold,
see e.g. \cite{FS2020} for the $hp$-approximation of Gevrey regular functions
in dimensions $d=2,3$. 
There, tensor products of univariate interpolants 
and of the presently developed ReLU NN emulations can be used.
For $\delta=1$ and $d=2,3$, a generalization of Theorem \ref{thm:reluhp} was shown in \cite{MOPS2020},
and for $d=2$ also in \cite{JOdiss}.
A generalization to multivariate Gevrey regular functions in norms without weights
(thus without algebraic point singularities)
was shown in \cite[Proposition 4.1]{OSZ2022hdim}.

\section{Conclusion and discussion}
\label{sec:Concl}

The present paper extended and improved emulation rate bounds 
from \cite{OPS2020} for
deep ReLU NNs for various spaces of continuous, piecewise polynomial
functions. 
In addition, due to the use of \Cheb polynomial expansions, 
all DNN emulation rate bounds are constructive.
Fast, \Cheb polynomial-based constructions are available which allow 
the linear construction of DNN approximations of continuous functions based
on point queries of functions in the \Cheb points.
This construction is a simple linear combination of 
``ReLU \Cheb features'', i.e. the presently constructed,
ReLU NN emulations of the \Cheb polynomials.
Quantitatively, 
for the emulation of (piecewise) polynomial functions 
of degree $p$ and relative accuracy $\eps\in (0,1]$,
the dependence of the network size on $p$ has been improved
from $C( p^2 + p \log(1/\eps) )$ in \cite{OPS2020} 
to $C( p (1+\log(p)) + p \log(1/\eps) )$.
In addition, as explained in Section~\ref{sec:conclemul}, 
the use of \Cheb polynomials for interpolation
brings improved numerical stability compared to the use of monomials,
and enables fast FFT-based NN constructions.
Compared to \cite{OPS2020}, 
we covered a wider range of Sobolev spaces
$W^{s,r}(I)$ for $0\leq s \leq 1$ and $1\leq r \leq \infty$
on arbitrary bounded intervals $I = (a,b)$,
keeping track of the natural scaling of Sobolev norms 
as a function of the interval length $b-a$.

While the present DNN emulation bounds addressed only deep, strict ReLU NNs, 
we emphasize that the results do generalize also to other activations.
The construction of ReLU NN approximations of \Cheb polynomials
in Section \ref{sec:basicnn} 
has been generalized to feedforward NNs with 
nonaffine, smooth activation functions 
in $C^2(U)\setminus\bbP_1$ for an open set $U\subset\R$,
in \cite[Corollary A.3 and Remark B.4]{OSX2024}.
We also remark that
the use of smooth activation functions, which are not as simple as the ReLU,
allows for better convergence rates.
In \cite{OSX2024},
univariate polynomials of degree $p$ are emulated by a NN of size bounded by $Cp$, 
for $C$ independent of the desired relative accuracy.
This is slightly better than the bound
$C( p (1+\log(p)) + p \log(1/\eps) )$ 
obtained here for strict ReLU NNs.
In terms of network size, 
the construction in \cite{OSX2024} is also more efficient than 
the shallow $\tanh$-NN approximation of all monomials of degree at most $p$
in \cite[Corollary 3.6]{DLM2021}.
For the case $d=1$,
the network constructed there emulates all univariate monomials of degree at most $p$
and has a network size bounded by $C p^2$.
An advantage of that construction 
is that the networks are shallow, i.e. they have fixed depth $2$ and width $O(p^2)$.
The depths of the NN emulations in this work and in \cite{OSX2024},
grow polylogarithmically with $p$.

Furthermore, in recent applications of DNN based numerical approximations
of solutions to PDEs, 
so-called \emph{``adaptive''} and \emph{``super-expressive''} activation functions
have been proposed.
Adaptive activations \cite{JKKAdAc2020}
introduce into each neuron a trainable scaling parameter,
which is updated during the training routine.
Numerical experiments in \cite{JKKAdAc2020}
report accelerated training and avoidance of bad local minima.
Because of the positive homogeneity of the ReLU,
i.e. $\varrho(\alpha x) = \alpha\varrho(x)$ for all $\alpha\geq0$ and $x\in\R$,
the introduction of scaling parameters inside neurons
does not affect the expressive power of the NNs.
All our results also apply to NNs with adaptive ReLU activations.
For recent theory on superexpressive activations, 
whose existence is stipulated by the Kolmogorov Superposition theorem, 
see e.g. \cite{ZShenEtAlFixNeuron,yarotsky2021elementary,shen2021neural}.
Arbitrary functions $f\in C^0([0,1]^d)$ for $d\in\N$ can be approximated 
by a NN whose depth and size depend only on $d$ and are independent of the desired accuracy.
This dramatic improvement in the bound on the network size
comes at the cost of needing high accuracy 
in the numerical evaluations of the affine transformations and nonlinear activation functions.
See e.g. \cite[Section 7]{YZ2020} 
for a discussion in a precursor of \cite{yarotsky2021elementary}.

Due to the equivalence of ReLU NNs and so-called 
\emph{spiking networks} which are prominent in neuromorphic computing 
(see \cite{stanojevic2022exact} and \cite[Theorem~3]{singh2023expressivity},
and the references there) 
all presently proved NN emulation rate bounds will imply corresponding
bounds for such networks.
By the construction in \cite{stanojevic2022exact},
for all ReLU NNs there exists a spiking NN with the same depth,
the same numbers of neurons per layer, the same realization 
and at most the same size, see the discussion in \cite[Section 6]{OSX2024}.


\end{document}